\RequirePackage[english]{babel}
\documentclass[a4paper,twoside,11pt]{amsart}
\usepackage[latin1]{inputenc}
\usepackage{amsmath}
\usepackage{amssymb}
\usepackage{amsfonts}
\usepackage{mathrsfs}
\usepackage{hyperref}
\usepackage{fixltx2e}[2005/12/01]

\usepackage{enumerate}
\renewcommand{\theenumi}{{\upshape{(\roman{enumi})}}}

\ifpdf
  \usepackage[matrix,arrow,cmtip,frame,graph,curve,2cell]{xy}
\else
  \usepackage[matrix,arrow,cmtip,frame,graph,curve,2cell,dvips]{xy}
\fi
\SelectTips{cm}{}
\newdir{(}{{}*!/-5pt/@^{(}}
\newdir{(x}{{}*!/-5pt/@_{(}}
\newdir{+}{{}*!/-9pt/{}}
\newdir{>+}{@{>}*!/-9pt/{}}
\entrymodifiers={+!!<0pt,\fontdimen22\textfont2>}
\UseAllTwocells

\usepackage[nosubsections]{mytheorems2}

\newtheoremstyle{dtheoremnopar}{3 mm}{1 mm}{\itshape}{}{\bfseries}{.}{ }
{\thmname{#1}\thmnumber{ #2}\thmnote{ \mdseries(#3)\bfseries}}

\theoremstyle{dtheoremnopar}
\newcounter{theoremx}
\newtheorem{theoremalpha}[theoremx]{Theorem}

\newcommand{\tref}[1]{\ref{#1}} 

\newcommand\N{\mathbb{N}}
\newcommand\Z{\mathbb{Z}}
\newcommand\QCoh{\mathbf{QCoh}} 

\newcommand{\pref}[1]{\eqref{#1}}

\newcommand\inj{\hookrightarrow}
\newcommand\surj{\twoheadrightarrow}
\newcommand\iso{\cong} 
\newcommand\map[3]{#1\colon #2\rightarrow #3}

\newcommand{\sA}{\mathcal{A}}
\newcommand{\sB}{\mathcal{B}}
\newcommand{\sF}{\mathcal{F}}
\newcommand{\sG}{\mathcal{G}}
\newcommand{\sH}{\mathcal{H}}
\newcommand{\sK}{\mathcal{K}}
\newcommand{\sL}{\mathcal{L}}
\newcommand{\sN}{\mathcal{N}}
\newcommand{\sO}{\mathcal{O}}
\newcommand{\sP}{\mathcal{P}}
\newcommand{\sQ}{\mathcal{Q}}
\newcommand{\sR}{\mathcal{R}}

\newcommand{\stX}{\mathscr{X}}
\newcommand{\sHom}{\mathcal{H}om}

\newcommand\id[1]{\mathrm{id}_{#1}} 
\newcommand\SG[1]{{\mathfrak{S}_{#1}}}   

\DeclareMathOperator{\Spec}{Spec}
\newcommand{\SEC}{\mathrm{SEC}}

\newcommand{\red}{\mathrm{red}}
\newcommand{\metale}{\mathrm{\acute{e}t}} 
\newcommand{\fp}{\mathrm{fp}} 
\newcommand\BGL{\mathrm{BGL}}
\newcommand\GL{\mathrm{GL}}
\newcommand{\A}[1]{\mathbb{A}^{#1}}    
\newcommand\BSG[1]{\mathrm{B}\SG{#1}}
\newcommand\Gm{\mathbb{G}_m}

\newcommand{\equalizer}[2]{\xymatrix@1@M=0mm@C=7mm{#1%
\ar@<.5ex>@{+->+}[r] \ar@<-.5ex>@{+->+}[r] & #2}}

\newcommand\loccit{\textit{loc.\ cit.}}
\newcommand{\etale}{\'{e}tale}
\newcommand{\Etale}{\'{E}tale}

\newcommand\weilr{\mathbf{R}} 
\newcommand\HilbSt{\mathscr{H}}

\newcommand{\AlgSp}{\mathbf{AlgSp}} 
\newcommand{\catC}{\mathbf{C}}
\newcommand{\catD}{\mathbf{D}}
\newcommand{\catE}{\mathbf{E}}
\newcommand{\catG}{\mathbf{G}}
\newcommand{\catHom}{\mathbf{Hom}}
\newcommand{\Stack}{\mathbf{Stack}}

\newcommand{\devissage}{d\'evissage}

\newcommand{\spref}[1]{\href{http://stacks.math.columbia.edu/tag/#1}{#1}}

\begin{document}

\title{Noetherian approximation of algebraic spaces and stacks}
\author{David Rydh}
\address{KTH Royal Institute of Technology, Department of Mathematics,
SE\nobreakdash-100\ 44\ Stockholm, Sweden}
\thanks{Supported by grant KAW 2005.0098 from the Knut and Alice Wallenberg
Foundation and by the Swedish Research Council 2008-7143 and 2011-5599.}
\email{dary@math.kth.se}
\date{2014-08-26}


\begin{abstract}
We show that every scheme (resp.\ algebraic space, resp.\ algebraic stack)
that is quasi-compact with quasi-finite diagonal can be approximated by a
noetherian scheme (resp.\ algebraic space, resp.\ stack). More generally, we
show that any stack which is \etale{}-locally a global quotient stack can be
approximated.
Examples of applications are generalizations of Chevalley's, Serre's and
Zariski's theorems and Chow's lemma to the non-noetherian setting.
We also show that every quasi-compact
algebraic stack with quasi-finite diagonal has a finite generically flat
cover by a scheme.
\end{abstract}

\maketitle


\setcounter{secnumdepth}{0}
\begin{section}{Introduction}
Let $A$ be a commutative ring and let $M$ be an $A$-module. Then $A$ is the
direct limit of its subrings that are finitely generated as
$\Z$-algebras and $M$ is the direct limit of its finitely generated
$A$-submodules. Thus, any affine scheme $X$ is an inverse limit of
affine schemes of finite type over $\Spec \Z$ and every quasi-coherent sheaf
on $X$ is a direct limit of quasi-coherent sheaves of finite type.

The purpose of this article is to give similar approximation results for
schemes, algebraic spaces and stacks, generalizing earlier results for schemes
by R.\ W.\ Thomason and T.\ Trobaugh~\cite[App.~C]{thomason-trobaugh}. We show,
for example, that every
quasi-compact and quasi-separated Deligne--Mumford stack $X$ can be written as
an inverse limit of Deligne--Mumford stacks $X_\lambda$ of finite type over
$\Spec \Z$.  Such results are sometimes known as ``absolute approximation''
\cite[App.~C]{thomason-trobaugh} in
contrast with the ``standard limit results'' \cite[\S8]{egaIV} which are
relative: given an
inverse limit $X=\varprojlim_\lambda X_\lambda$, describe \emph{finitely
  presented} objects over $X$ in terms of finitely presented objects over
$X_\lambda$ for sufficiently large $\lambda$.

We say that an algebraic stack $X$ is of \emph{global type} if
\etale{}-locally
$X$ is a global quotient stack, cf.\ Section~\ref{S:global-type} for a precise
definition.
Examples of stacks of global type are quasi-compact and quasi-separated
schemes, algebraic spaces, Deligne--Mumford stacks and
algebraic stacks with quasi-finite (and locally separated) diagonals.
For convenience, we also introduce the notion of approximation type. Every
stack of global type is of approximation type
(Proposition~\ref{P:approx-type-sorites}).

The main result of this paper, Theorem~\tref{T:APPROXIMATION}, is that any
stack of approximation type can be approximated by a
noetherian stack. More generally, if $X\to S$ is a morphism between stacks of
approximation type, then $X$ is an inverse limit of finitely
presented stacks over~$S$. 

The primary application of the approximation theorem is the elimination of
noetherian, excellency and finiteness hypotheses. When eliminating noetherian
hypotheses in
statements about finitely presented morphisms $X\to Y$ which are \emph{local}
on $Y$, the basic affine approximation result referred to in the beginning is
sufficient, cf.\ the standard limit results in~\cite[\S8]{egaIV} and
Appendix~\ref{A:std-limits}.
For \emph{global} problems, it is crucial to have
Theorem~\tref{T:APPROXIMATION}. Examples of such applications, including
generalizations of Chevalley's, Serre's and Zariski's theorems and Chow's
lemma, are given in Section~\ref{S:applications}.
Although this paper is written with stacks in mind, most of the applications in
\S\ref{S:applications} are new also when applied to schemes and algebraic
spaces. We also answer a question by Grothendieck~\cite[Rem.\ 18.12.9]{egaIV}
on integral morphisms affirmatively, cf.\
Theorem~\pref{T:char-of-integral-morphisms}.

Before stating the main results we need another definition. An algebraic stack
$X$ has the \emph{completeness property} if every quasi-coherent sheaf on $X$
is a filtered direct limit of finitely presented sheaves or, equivalently, if
the abelian category $\QCoh(X)$ is compactly generated. An algebraic stack $X$
is
\emph{pseudo-noetherian} if it is quasi-compact, quasi-separated and $X'$ has
the completeness property for every finitely presented morphism $X'\to X$ of
algebraic stacks.
A key insight during this work was that, for approximation purposes,
``pseudo-noetherian'', rather than the ``completeness property'', is the
correct notion to work with.

\begin{theoremalpha}[Completeness]\label{T:COMPLETENESS}
Every stack of approximation type is
pseudo-noetherian.
\end{theoremalpha}

Every noetherian stack is pseudo-noetherian~\cite[Prop.~15.4]{laumon}. In the
category of schemes, Theorem~\tref{T:COMPLETENESS} is
well-known~\cite[\S6.9]{egaI_NE} and a slightly weaker result for algebraic
spaces is due to Raynaud and Gruson~\cite[Prop.~5.7.8]{raynaud-gruson}.

\begin{theoremalpha}[Finite coverings]\label{T:FINITE-COVERINGS}
Let $X$ be a quasi-compact stack with quasi-finite and separated diagonal
(resp.\ a quasi-compact Deligne--Mumford stack with quasi-compact and separated
diagonal). Then
there exists a scheme $Z$ and a finite, finitely presented and surjective
morphism $Z\to X$ that is flat (resp.\ \etale{}) over a dense quasi-compact
open substack $U\subseteq X$.
\end{theoremalpha}

When $X$ is a \emph{noetherian Deligne--Mumford stack},
Theorem~\tref{T:FINITE-COVERINGS} is due to G.\ Laumon and L.\
Moret-Bailly~\cite[Thm.~16.6]{laumon}. When $X$ is of finite type over a
noetherian scheme, the existence of a scheme $Z$ and a finite and surjective,
but not necessarily generically flat,
morphism $Z\to X$ was shown by D.\ Edidin, B.\ 
Hassett, A.\ Kresch and A.\
Vistoli~\cite[Thm.~2.7]{edidin-etal_Brauer-groups-and-quotient-stacks}.


Before stating the main approximation theorem, we introduce the approximation
of various properties. This allows us to unify various approximation results
for schemes, algebraic spaces, Deligne--Mumford stacks and so on, in one
theorem. Thus, consider the following properties of a morphism of algebraic
stacks:
\begin{enumerate}
\renewcommand{\theenumi}{{\upshape{(PA)}}}%
\item affine; quasi-affine; representable; separated; locally separated
  (i.e., diagonal is an immersion); separated diagonal; locally separated
  diagonal; unramified diagonal; quasi-finite diagonal; affine diagonal;
  quasi-affine diagonal; finite inertia; abelian inertia; tame inertia (i.e.,
  stabilizer groups are finite and linearly reductive);\label{I:prop-affine}
\renewcommand{\theenumi}{{\upshape{(PC)}}}%
\item closed immersion; immersion; monomorphism of finite type;
  unramified; quasi-finite; finite; proper with finite diagonal;
  \label{I:prop-closed}
\renewcommand{\theenumi}{{\upshape{(PI)}}}%
\item integral.\label{I:prop-integral}
\end{enumerate}
These properties are all stable under composition and fppf-local on the
target. Also note that affine morphisms have all properties in
\ref{I:prop-affine} and closed immersions have all properties in
\ref{I:prop-closed}.

\begin{theoremalpha}[Approximation of properties]\label{T:APPROXIMATION-PROPERTIES}
Let $S$ be a quasi-compact algebraic stack and let $\{X_\lambda\to S\}$ be an
inverse system of quasi-compact and quasi-separated morphisms of algebraic
stacks with affine bonding maps
$X_\mu\to X_\lambda$ and limit $X\to S$.
\begin{enumerate}
\item Let $P$ be one of the
properties in~\ref{I:prop-affine}.
Then $X\to S$ has property~$P$ if and only if there exists an index $\alpha$
such that $X_\lambda\to S$ has property~$P$ for every $\lambda\geq\alpha$.
\item Assume that the morphisms $X_\lambda\to S$ are of finite type and the
bonding maps $X_\mu\to X_\lambda$ are closed immersions; hence $X\to S$ is of
finite type. Let $P$ be one of the
properties in~\ref{I:prop-closed}.
Then $X\to S$ has property~$P$ if and only if there exists an index $\alpha$
such that $X_\lambda\to S$ has property~$P$ for every $\lambda\geq\alpha$.
\item Assume that $S$ is quasi-separated.
Then $X$ is a scheme if and only if there exists an index $\alpha$
such that $X_\lambda$ is a scheme for every $\lambda\geq\alpha$.\label{TI:C-scheme}
\end{enumerate}
\end{theoremalpha}

Note that, contrarily to similar results, we do not require that
$X_\lambda\to S$ is of finite presentation in
Theorem~\tref{T:APPROXIMATION-PROPERTIES}. This is crucial for
Theorem~\tref{T:APPROXIMATION}~\ref{TI:approx-absolute-properties} and
many applications, e.g., Remark~\pref{R:Totaros-theorem}.

\begin{theoremalpha}[Approximation]\label{T:APPROXIMATION}
Let $S$ be a pseudo-noetherian algebraic stack and let $X\to S$ be a
morphism of approximation type (these assumptions are satisfied if
$X$ and $S$ are of global type).
Then, there exists
a \emph{finitely presented} morphism $X_0\to S$ and an \emph{affine}
$S$-morphism $X\to X_0$. Moreover, $X\to X_0\to S$ can be chosen such that
the following holds.
\begin{enumerate}
\item If $X\to S$ is of finite type, then $X\to X_0$ is a closed immersion.
\label{TI:B-first}
\item If $X\to S$ has one of the properties in
\ref{I:prop-affine}, \ref{I:prop-closed} or \ref{I:prop-integral},
then so has $X_0\to S$.
\item If $X\to \Spec \Z$ has one of the properties in
\ref{I:prop-affine} then so has $X_0\to \Spec \Z$. If
$X$ is a scheme then so is $X_0$.
\label{TI:approx-absolute-properties}
\label{TI:B-last}
\end{enumerate}
Furthermore, $X$ can be written as an inverse limit
$\varprojlim_\lambda X_\lambda$
of finitely presented $S$-stacks with affine bonding maps such that for
every $\lambda$, the factorization $X\to X_\lambda\to S$ satisfies
\ref{TI:B-first}--\ref{TI:B-last} with $X_0=X_\lambda$. Finally, if $X\to S$
is of finite type (resp.\ integral), then there is such an inverse system
with bonding maps that are closed immersions (resp.\ finite).
\end{theoremalpha}

When $X$ and $S$ are \emph{schemes}, parts of
Theorems~\tref{T:APPROXIMATION-PROPERTIES} and~\tref{T:APPROXIMATION} have been
shown by R.\ W.\ Thomason and T.\ Trobaugh~\cite[App.~C]{thomason-trobaugh},
B.\ Conrad~\cite[Thm.~4.3,\ App.~A]{conrad_nagata}
and M.\ Temkin~\cite[Thm.~1.1.2]{temkin_relative-RZ-spaces}.
%
When $X$ and $S$ are \emph{algebraic spaces}, parts of
Theorem~\tref{T:APPROXIMATION} were recently obtained independently by
B.\ Conrad, M.\ Lieblich and
M.\ Olsson~\cite[\S3]{conrad-lieblich-olsson_Nagata}.

There are also some approximation results for group schemes. If $G$ is a
quasi-compact group scheme over a field, then D.\ Perrin has shown that $G$ is
an inverse limit of group schemes of finite type~\cite{perrin_approx-groups}.

\begin{subsection}{\Etale{} \devissage{}}
The \etale{} \devissage{} method of~\cite{rydh_etale-devissage} is the primary
technique behind the proofs of Theorems~\tref{T:COMPLETENESS}
and~\tref{T:APPROXIMATION} and is prominent in all previous treatments of
approximation for algebraic spaces. There is a subtle, yet crucial, difference
in our treatment, though. Our \devissage{} gives statements of the form: given
$X'\to X$ surjective and \etale{}, then $X$ can be approximated if $X'$ can be
approximated (Proposition~\ref{P:pseudo-noeth:et} and
Lemma~\ref{L:approximation:etale}). This is accomplished by reducing to the
case where $X'\to X$ is an \etale{} neighborhood. All other treatments, from
Raynaud--Gruson and onwards, would demand that $X'\to X$ is an \etale{}
neighborhood such that $X'$ is (quasi-)affine. This approach has recently been
formalized as ``scallop decompositions'' by J.\ Lurie and can only deal with
algebraic spaces. As our inductive approach is not based on quasi-affine
neighborhoods, we have to be more careful when formulating the inductive
hypothesis, e.g., use ``pseudo-noetherian'' instead of the
``completeness property''.
\end{subsection}

\begin{subsection}{Overview}
We begin with some conventions on stacks in Section~\ref{S:notation}. In
Section~\ref{S:global-type}, we define stacks of global type and stacks
of approximation type. We show that every
quasi-compact algebraic stack with quasi-finite and locally separated diagonal
is of
global type.
In Section~\ref{S:etale-devissage}, we briefly outline the \etale{}
\devissage{} method.
In Sections~\ref{S:completeness} and~\ref{S:finite-coverings}, we prove
Theorems~\tref{T:COMPLETENESS} and~\tref{T:FINITE-COVERINGS}.
In Section~\ref{S:approximation-properties}, we prove
Theorem~\tref{T:APPROXIMATION-PROPERTIES} \emph{when the morphisms
$X_\lambda\to S$ are of finite presentation}. This case of the theorem is
essentially independent of
Theorems~\tref{T:COMPLETENESS} and~\tref{T:FINITE-COVERINGS}.
In Section~\ref{S:approximation}, we prove the general form of
Theorems~\tref{T:APPROXIMATION-PROPERTIES} and~\tref{T:APPROXIMATION}.
We conclude with numerous applications of the main theorems in
Section~\ref{S:applications}.

In the appendices, we extend the standard results~\cite[\S8--9]{egaIV} on limits
and constructible properties from schemes to stacks.
\end{subsection}

\begin{subsection}{Acknowledgments}
I would like to thank J.\ Alper, B.\ Conrad, P.\ Gross, J.\ Hall, M.\ Lieblich,
M.\
Olsson, R.\ Skjelnes and M.\ Temkin for useful comments and discussions. In
particular, I am grateful to B.\ Conrad for suggesting the adjective
``pseudo-noetherian'' which I find very apt.
\end{subsection}

\end{section}
\setcounter{secnumdepth}{3}

\setcounter{tocdepth}{1}
\tableofcontents

\begin{section}{Stack conventions}\label{S:notation}
We follow the conventions in~\cite{laumon} except that we
follow~\cite[\spref{026O}]{stacks-project} and do not require that
the diagonal of an algebraic stack is quasi-compact and separated. One reason
for this is
that stacks with non-separated diagonals naturally appears
in the context of~\cite{rydh_etale-devissage}. On the other hand, very little
is lost by assuming that all algebraic stacks are \emph{quasi-separated},
i.e., that the diagonal is
quasi-compact and quasi-separated, and most results require this
hypothesis.
If
$Y$ is a quasi-compact and quasi-separated algebraic stack, then $X\to Y$ is
quasi-compact if and only if $X$ is so. In particular, an open substack
$U\subseteq Y$ is quasi-compact if and only if the morphism $U\to Y$ is
quasi-compact.

\begin{xpar}
A \emph{presentation} of a stack $X$ is an algebraic space $X'$ and a
faithfully flat morphism $X'\to X$ locally of finite presentation.
A morphism $\map{f}{X}{Y}$ of stacks is
\emph{representable} (resp.\ \emph{strongly representable}) if $X\times_Y Y'$
is an algebraic space (resp.\ a scheme) for every scheme $Y'$ and morphism
$Y'\to Y$. Note that the property of being representable is fppf-local on the
target. Indeed, a morphism is representable if and only if its diagonal is a
monomorphism. This is not the case for the property of being strongly
representable. A morphism $X\to S$ of stacks is \emph{locally separated} if the
diagonal $\Delta_{X/S}$ is an immersion. In particular, every locally separated
morphism is representable.

An algebraic stack $X$ is \emph{Deligne--Mumford} if there exists an \etale{}
presentation of $X$, or equivalently, if the diagonal is
unramified~\cite[Thm.~8.1]{laumon}, \cite[\spref{06N3}]{stacks-project}.

An algebraic stack is \emph{noetherian} if it is quasi-compact and
quasi-separated and admits a noetherian presentation.
\end{xpar}

\begin{xpar}[Unramified and \etale{}]
For the definition and general properties of unramified and \etale{} morphisms
of stacks, we refer to~\cite[App.~B]{rydh_embeddings-of-unramified}. In
particular, by an unramified morphism we mean a formally unramified morphism
that is locally of \emph{finite type} (not necessarily of finite
presentation). An \etale{} morphism is a formally \etale{} morphism that is
locally of finite presentation. A morphism is unramified if and only if it is
locally of finite type and its diagonal is \etale{}. A morphism is \etale{} if
and only if it is locally of finite presentation, unramified and flat. We do
not require that unramified and \etale{} morphisms are representable.
\end{xpar}

\begin{xpar}[Quasi-finite]
A morphism $\map{f}{X}{Y}$ of stacks is (locally) \emph{quasi-finite} if $f$ is
(locally) of finite
type, every fiber of $f$ is discrete and every fiber of the diagonal $\Delta_f$
is discrete. Equivalently, $f$ is (locally) quasi-finite if and only if $f$ is
(locally) of finite
type, every fiber of $f$ is zero-dimensional and every fiber of the diagonal of
$f$ is zero-dimensional.
%
%
Note that the diagonal of a quasi-finite and quasi-separated morphism is
quasi-finite.
\end{xpar}

Given a morphism $\map{f}{X}{Y}$ of algebraic stacks, we say that
$\map{h\circ g}{X}{X_0\to Y}$ is a \emph{factorization} of $f$ if there exists
a $2$-isomorphism $f\Rightarrow h\circ g$.

The \emph{inertia stack} of a morphism $\map{f}{X}{Y}$ of algebraic stacks is
the algebraic stack $I_{X/Y}:=X\times_{X\times_Y X} X$. It
comes with a representable morphism $\map{I_f}{I_{X/Y}}{X}$ equipped with the
structure of a relative group space over $X$.
We say that $f$ has finite (etc.)\ inertia if $I_f$ is finite
(etc.).

By convention, all our inverse systems are filtered and all maps in
inverse systems are affine.
\end{section}


\begin{section}{Stacks of global type and approximation type}\label{S:global-type}
In this section, we define stacks of global type and show that every
quasi-compact stack with quasi-finite and locally separated diagonal is
of global type.
We also define stacks of approximation type, which is a natural class
of stacks for our purposes. Every stack of global type and every stack of
finite presentation over a quasi-compact and quasi-separated scheme or
algebraic space is of
approximation type.

\begin{definition}\label{D:global-type}
Let $X$ be an algebraic stack. We say that $X$ is
\begin{enumerate}
\item \emph{basic} if $X=[V/\GL_n]$ for
  some \emph{quasi-affine} scheme $V$ and integer~$n$;
\item of \emph{global type} if there exists a
  representable, \etale{}, finitely presented and surjective morphism
  $\map{p}{X'}{X}$ such that $X'$ is basic;
\item of \emph{s-global type} if there exists a separated,
  representable, \etale{}, finitely presented and surjective morphism
  $\map{p}{X'}{X}$ such that $X'$ is basic;
\item a \emph{global quotient stack} if $X=[V/\GL_n]$
  where $V$ is an algebraic space.
\end{enumerate}
\end{definition}

\begin{remark}[Relation with the resolution property]\label{R:Totaros-theorem}
B.\ Totaro has shown that a normal noetherian stack is a basic stack
if and only if it has the resolution
property~\cite{totaro_resolution-property}. By recent work of P.\ Gross, this
also holds for non-normal noetherian stacks~\cite[Thm.~6.3.1]{gross_thesis}.
Using Theorem~\tref{T:APPROXIMATION-PROPERTIES}, one can give a
very satisfactory proof of this result that is also valid without noetherian
hypotheses~\cite{gross_tensor-generators}.
Thus, a stack $X$ is of global type if and only if the
resolution property holds \etale{}-locally on $X$.

Every basic stack has affine diagonal. There are very few examples of stacks
with affine diagonal that are known to be non-basic. For example, there is not
a single example of a
non-basic separated scheme or Deligne--Mumford
stack~\cite[Question~1]{totaro_resolution-property} although S.\ Payne has
given some evidence that a certain proper toric three-fold is not
basic~\cite{payne_toric-vect-bundles-res-prop}.
\end{remark}

The following example shows that many global quotient stacks are of global
type.

\begin{example}
Let $G$ be an affine smooth group scheme over $\Spec \Z$ with connected
fibers, e.g., $G=\GL_{n,\Z}$. If $X$ is a \emph{normal noetherian scheme} with
an action of $G$, then $[X/G]$ is of s-global type.
%
%
Indeed, if $X$ is quasi-projective, then $[X/G]$ has the resolution property
by~\cite[Thm.~2.1 (2)]{totaro_resolution-property} and a result of
Sumihiro~\cite[Thm.~3.8]{sumihiro_II} states that there exists a Zariski
covering $U\to [X/G]$ such that $X\times_{[X/G]} U$ is quasi-projective.
\end{example}

\begin{question}
Is every noetherian global quotient stack of s-global type?
\end{question}

\begin{remark}
Note that being of global type is not a purely local condition since
$\map{p}{X'}{X}$ is required to be of finite presentation, so $X$ is
quasi-compact and quasi-separated. We require that $p$ is representable
as currently there is no suitable \devissage{} for non-representable \etale{}
morphisms.

If $X$ is a stack of global type, then $X$ is quasi-compact, quasi-separated
and $\Delta_X$ is locally separated with affine fibers,
cf.~\cite[App.~A]{rydh_etale-devissage}.
If $X$ is a stack of s-global type, then $\Delta_X$ is also quasi-affine.
%
%
Note that there exist stacks of global type with non-separated diagonals, e.g.,
every quasi-compact and quasi-separated Deligne--Mumford stack is of global
type.
%
\end{remark}

I am not aware of any example of a stack with quasi-affine diagonal that is not
of global type. There are, on the other hand, examples of stacks with
quasi-affine diagonal that are not global quotient stacks. One such example,
with affine diagonal, is the $\Gm$-gerbe over a complex surface $Y$
corresponding to a non-torsion element of the cohomological Brauer group
$H^2_{\metale}(Y,\Gm)$~\cite[Ex.~3.12]{edidin-etal_Brauer-groups-and-quotient-stacks}. Another
example is the two-dimensional Deligne--Mumford stack
of~\cite[Ex.~2.21]{edidin-etal_Brauer-groups-and-quotient-stacks} which has
quasi-affine diagonal. These two examples are easily seen to be of global
type. A third example of a stack of global type that is not a global quotient
stack, albeit not with quasi-affine diagonal, is the stack
$\mathfrak{M}_0^{\leq m}$ of prestable curves of genus~$0$ with at most $m$
nodes, for $m\geq 2$~\cite[\S5]{kresch_flat-strat}.

\begin{proposition}\label{P:finite-flat-pres=>basic}
Let $X$ be an algebraic stack with a finite flat presentation $\map{p}{V}{X}$
such that $V$ is quasi-affine. Then $X$ is a basic stack.
\end{proposition}
\begin{proof}
Let
$\sL=p_*\sO_V$ which is a locally free sheaf of finite rank. Let
$\map{x}{\Spec k}{X}$ be a point and let $G_x$ be the stabilizer group
scheme of $x$. The stabilizer group scheme acts on the $k$-vector space
$\sL_x$ and this coaction is faithful since the stabilizer action on the
subscheme $V_x\inj \mathbf{V}(\sL_x)$ is free. Replacing $\sL$ with the direct
sum of $\sL$ and a free sheaf, we can further assume that $\sL$ is of constant
rank $r$.

Let $Z=\underline{Isom}_X(\sO_X^r,\sL)\subseteq \mathbf{V}\bigl((\sL^\vee)^r\bigr)$
be the frame bundle of $\sL$. The morphism $Z\to X$ is a
$\GL_{r,\Z}$-torsor and $Z$ is an algebraic space since the action of
$G_x$ on the fiber $Z_x$ is free. Since $Z\to X$ is affine, we have that
$Z\times_X V$ is quasi-affine. As $Z\times_X V\to Z$ is a finite flat
presentation, we conclude that $Z$ is quasi-affine as well
by~\cite[Lem.~C.1]{rydh_etale-devissage}.
Thus $X=[Z/\GL_r]$ is a basic stack.
\end{proof}

\begin{corollary}\label{C:qcqf=>global-type}
Every quasi-compact algebraic stack with quasi-finite and locally separated
(resp.\ separated) diagonal is
of global type (resp.\ s-global type).
\end{corollary}
\begin{proof}
Let $X$ be a stack with quasi-finite and locally separated (resp.\ separated)
diagonal. By~\cite[Thm.~7.2]{rydh_etale-devissage}, there exists a
representable (resp.\ representable and separated) \etale{} surjective morphism
$X'\to X$ of finite presentation and a finite
flat presentation $V\to X'$ with $V$ quasi-affine. Since $X'$ is a basic
stack~(Proposition~\ref{P:finite-flat-pres=>basic}), we have, by definition,
that $X$ is of global type (resp.\ s-global type).
\end{proof}

The following result, which partly generalizes the two previous results but
depends on~\cite{gross_tensor-generators}, is not used in this paper but
included for completeness. In particular, it justifies the usage of \etale{} in
the definition of s-global type.

\begin{proposition}
Let $\map{f}{X}{Y}$ be a morphism of algebraic stacks.
\begin{enumerate}
\item\label{PI:glob-type:qaff-stable}
  Assume that $f$ is quasi-affine. If $Y$ is basic (resp.\ of s-global
  type, resp.\ of global type), then so is $X$.
\item\label{PI:basic:ff-descends}
  Assume that $f$ is finite and faithfully flat of finite presentation.
  If $X$ is basic, then so is $Y$.
\item\label{PI:s-glob-type:qfsep-descends}
  Assume that $f$ is quasi-finite, representable, separated
  and faithfully flat of
  finite presentation. If $X$ is of s-global type, then so is $Y$.
\end{enumerate}
In particular, in the definition of s-global type, we can replace ``\etale''
with ``quasi-finite and flat''.
\end{proposition}
\begin{proof}
\ref{PI:glob-type:qaff-stable} is trivial from the definitions.
\ref{PI:basic:ff-descends} follows from~\cite[Prop.~4.3
  (vii)]{gross_tensor-generators} taking into
account~\cite[Cor.~5.9]{gross_tensor-generators}. For
\ref{PI:s-glob-type:qfsep-descends}, we can assume that $X$ is basic. Then
using \cite[Thm.~6.3 (i)]{rydh_etale-devissage} and
\ref{PI:glob-type:qaff-stable}, we can assume that $f$ is finite and the result
follows from \ref{PI:basic:ff-descends}.
\end{proof}

\begin{definition}\label{D:approx-type}
We say that a morphism $\map{f}{X}{Y}$ of algebraic stacks is of \emph{strict
approximation type} if $f$ can be written as a composition of affine morphisms
and finitely presented morphisms. We say that $f$ is of \emph{approximation
type} if there exists a surjective representable and finitely presented
\etale{} morphism $\map{p}{X'}{X}$ such that $f\circ p$ is of strict
approximation type. We say that an algebraic stack $X$ is of (strict)
approximation type if $X\to \Spec \Z$ is of (strict) approximation type.
\end{definition}

We begin with the usual sorites of morphisms of (strict) approximation type.

\begin{proposition}\label{P:approx-type-sorites}\; ---
\begin{enumerate}
\item Finitely presented morphisms and quasi-affine morphisms are of strict
approximation type.\label{PI:at-sorites:fp/aff}
\item Every morphism of approximation type is quasi-compact and quasi-separated.
\label{PI:at-sorites:qcqs}
\item Basic stacks are of strict approximation type.
Stacks of global type, e.g., quasi-compact and quasi-separated
Deligne--Mumford stacks, are of approximation type.
\label{PI:at-sorites:glob-type}
\item If $\map{f}{X}{Y}$ is of (strict) approximation type and
$Y'\to Y$ is a morphism, then $\map{f'}{X\times_Y Y'}{Y'}$ is of (strict)
approximation type.\label{PI:at-sorites:base-change}
\item If $\map{f}{X}{Y}$ and $\map{g}{Y}{Z}$ are of (strict) approximation
type then so is $g\circ f$.\label{PI:at-sorites:composition}
\item If $\map{f_1}{X_1}{Y_1}$ and $\map{f_2}{X_2}{Y_2}$ are of (strict)
approximation type then so is $f_1\times f_2$.\label{PI:at-sorites:product}
\item If $\map{f}{X}{Y}$ and $\map{g}{Y}{Z}$ are morphisms such that
$g\circ f$ and $\Delta_g$ are of (strict) approximation type, then so is
$f$.\label{PI:at-sorites:first-of-comp}
\item If $\map{f}{X}{Y}$ is of (strict) approximation type, then so is
$\map{\Delta_f}{X}{X\times_Y X}$.\label{PI:at-sorites:diagonal}
\end{enumerate}
In particular, morphisms between stacks of global type are of approximation
type.
\end{proposition}
\begin{proof}
\ref{PI:at-sorites:fp/aff}, \ref{PI:at-sorites:qcqs},
\ref{PI:at-sorites:base-change}, \ref{PI:at-sorites:composition}
and~\ref{PI:at-sorites:product} are obvious
and~\ref{PI:at-sorites:first-of-comp} follows from a standard argument.

\ref{PI:at-sorites:glob-type} Let $X=[V/\GL_n]$ with $V$ quasi-affine. Then
there is an induced quasi-affine morphism $X\to \BGL_{n,\Z}$, so $X$ is of
strict approximation type.

\ref{PI:at-sorites:diagonal} If $f=f_n\circ f_{n-1}\circ\dots\circ f_1$ is a
composition of affine morphisms and finitely presented morphisms, then so is
$\Delta_f$ since it is a composition of pull-backs of the
$\Delta_{f_i}$'s. If $\map{p}{X'}{X}$ is a surjective representable and
finitely presented \etale{} morphism such that $f\circ p$ is of strict
approximation type, then $\Delta_f\circ p=(p\times p)\circ \Delta_{f\circ p}$
is of strict approximation type.
\end{proof}

\begin{proposition}\label{P:approximation:etale-descent}
Let $\map{f}{X}{Y}$ be a morphism of algebraic stacks.
\begin{enumerate}
\item Let $\map{g}{Y'}{Y}$ be surjective, representable and \etale{} of finite
presentation. If $\map{f'}{X\times_Y Y'}{Y'}$ is of approximation type, then so
is $f$.
\item Let $\map{p}{X'}{X}$ be surjective, representable and \etale{} of finite
presentation. If $f\circ p$ is of approximation type, then so is $f$.
\end{enumerate}
\end{proposition}
\begin{proof}
Immediate from the definition of approximation type.
\end{proof}

We will now give an analogue of
Proposition~\pref{P:approximation:etale-descent} for finite flat
morphisms which also is valid for strict approximation type.

\begin{proposition}\label{P:strict-approximation:finite-flat-descent}
Let $\map{f}{X}{Y}$ be a morphism of algebraic stacks.
\begin{enumerate}
\item Let $\map{g}{Y'}{Y}$ be finite and faithfully flat of finite presentation.
If $\map{f'}{X\times_Y Y'}{Y'}$ is of (strict) approximation type, then so
is $f$.
\item Let $\map{p}{X'}{X}$ be finite and faithfully flat of finite presentation
and assume that $X$ is quasi-compact. If $f\circ p$ is of (strict) approximation
type, then so is $f$.
\end{enumerate}
\end{proposition}
\begin{proof}
(i) We will make use of the Weil restriction of stacks along
$g$~\cite[\S3]{rydh_hilbert}. Recall that the Weil restriction is a
$2$-functor $\map{\weilr_g}{\Stack_{/Y'}}{\Stack_{/Y}}$ from stacks over $Y'$
to stacks over $Y$. If $Z$ is an algebraic stack over $Y'$, then so is 
$\weilr_g(Z)\to Y$. If $\map{h}{Z_1}{Z_2}$ is a morphism of algebraic
stacks over $Y'$, then there is an induced morphism
$\map{\weilr_g(h)}{\weilr_g(Z_1)}{\weilr_g(Z_2)}$. If $h$ is affine
(resp.\ of finite presentation, resp.\ \etale{} and surjective,
resp.\ representable, resp.\ a monomorphism)
then so is $\weilr_g(h)$. In particular, if $f'$ is of (strict) approximation
type then so is $\map{\weilr_g(f')}{\weilr_g(X\times_Y Y')}{\weilr_g(Y')=Y}$.
As $\weilr_g$ is the right adjoint to the pull-back functor
$\map{g^{-1}}{\Stack_{/Y}}{\Stack_{/Y'}}$ we have unit and counit maps
\begin{align*}
\map{\eta&}{X}{\weilr_g(X\times_Y Y')} \\
\map{\epsilon&}{\weilr_g(X\times_Y Y')\times_Y Y'}{X\times_Y Y'}
\end{align*}
such that
$$f=\weilr_g(f')\circ \eta$$
$$\weilr_g(f')\times_Y \id{Y'}=f'\circ \epsilon$$
$$\id{X\times_Y Y'}=\epsilon\circ (\eta\times_Y \id{Y'}).$$
It is thus enough to show that $\eta$ is of (strict)
approximation type. Since $\weilr_g(f')$ and $f'$ are of (strict) approximation
type, so are $\epsilon$ and $\eta\times_Y \id{Y'}$. We further observe that
if $f$ is arbitrary (resp.\ representable, resp.\ a monomorphism) then so
is $\epsilon$ and it follows that $\eta$ is representable (resp.\ a
monomorphism, resp.\ an isomorphism). We can thus replace $f$ with $\eta$ and
$\map{g}{Y'}{Y}$ with its pull-back
$\weilr_g(X\times_Y Y')\times_Y Y'\to \weilr_g(X\times_Y Y')$
and further assume that $f$ is representable (resp.\ a monomorphism,
resp.\ an isomorphism). Repeating the argument twice settles (i).

(ii) We will make use of the \emph{Hilbert stack of points}~\cite{rydh_hilbert}.
Since $X$ is quasi-compact, and the fiber rank of $p$ is locally constant,
we can assume that $p$ has constant rank $d$. This induces a morphism
$X\to \HilbSt^d_{Y/Y}$ and a cartesian diagram
$$\xymatrix{X'\ar[r]\ar[d] & Z\ar[d]\\
X\ar[r] & \HilbSt^d_{Y/Y}\ar[r]\ar@{}[ul]|\square & Y}$$
where $Z\to \HilbSt^d_{Y/Y}$ is the universal finite flat family of constant
rank $d$. Since $\HilbSt^d_{Y/Y}\to Y$ is of finite presentation, we have
that $X'\to Z$ is of (strict) approximation type, so $X\to Y$ is
of (strict) approximation type by (i).
\end{proof}

For our immediate purposes, we only need
Proposition~\pref{P:strict-approximation:finite-flat-descent} for finite
\emph{\etale} coverings. The proof can then be simplified by replacing the Weil
restriction and the Hilbert stack with symmetric products and $\BSG{d}\times Y$.
However, Proposition~\pref{P:strict-approximation:finite-flat-descent} together
with~\cite[Thm.~6.3 (ii)]{rydh_etale-devissage} shows that in the definition of
approximation type, we can essentially replace \etale{} with quasi-finite flat.
To be precise, if $\map{f}{X}{Y}$ is a morphism of algebraic stacks such that
$X$ is quasi-compact and quasi-separated, then $f$ is of approximation type
if and only if there exists a representable, locally separated, quasi-finite
and faithfully flat
morphism $\map{p}{X'}{X}$ of finite presentation such that $f\circ p$ is
of strict approximation type.

It should be noted that the only reason for insisting upon $p$ being
representable in Definitions~\pref{D:global-type} and~\pref{D:approx-type} is
that currently we only have
a nice \etale{} \devissage{} for representable morphisms. This is also the
reason
behind the local separatedness assumption in
Corollary~\pref{C:qcqf=>global-type}.

\begin{questions}
Is the notion of (strict) approximation type fppf-local on the target? Is every
noetherian stack of approximation type? Is every quasi-compact and
quasi-separated stack of approximation type?
Is every representable morphism of approximation type?
Is every quasi-compact stack
with quasi-affine diagonal of global type?
\end{questions}

\end{section}


\begin{section}{\Etale{} \devissage{}}\label{S:etale-devissage}
The \etale{} \devissage{} method reduces questions about general \etale{}
morphisms
to \etale{} morphisms of two basic types. The first type is finite \etale{}
coverings. The second is \emph{\etale{} neighborhoods} or equivalently
pushouts of \etale{} morphisms and open immersions --- the \etale{} analogue of
open
coverings consisting of two open subsets. For the reader's convenience, we
summarize the main results of~\cite{rydh_etale-devissage}.

\begin{definition}
Let $X$ be an algebraic stack and let $Z\inj |X|$ be a closed subset.
An \etale{} morphism $\map{p}{X'}{X}$ is an \emph{\etale{} neighborhood of $Z$}
if $p|_{Z_\red}$ is an isomorphism.
\end{definition}

\begin{theorem}[{\cite[Thm.~A]{rydh_etale-devissage}}]
\label{T:etnbhd-descent:QCoh}
Let $X$ be an algebraic stack and let $U\subseteq X$ be an open substack. Let
$\map{f}{X'}{X}$ be an \etale{} neighborhood of $X\setminus U$ and let
$U'=f^{-1}(U)$. The natural functor
$$\map{(|_U,f^*)}{\QCoh(X)}{\QCoh(U)\times_{\QCoh(U')}
\QCoh(X')}$$
is an equivalence of categories.
\end{theorem}

\begin{theorem}[{\cite[Thm.~B]{rydh_etale-devissage}}]
\label{T:etnbhd-is-pushout}
Let $X$ be an algebraic stack and let $\map{j}{U}{X}$ be an open immersion.
Let $\map{p}{X'}{X}$ be an \etale{} neighborhood of $X\setminus U$
and let $\map{j'}{U'}{X'}$ be the pull-back of $j$. Then $X$ is the pushout
in the category of algebraic stacks of $p|_U$ and $j'$.
\end{theorem}

\begin{theorem}[{\cite[Thm.~C]{rydh_etale-devissage}}]\label{T:etnbhd-pushout}
Let $X'$ be a quasi-compact and quasi-separated algebraic stack, let
$\map{j'}{U'}{X'}$ be a quasi-compact open immersion and let $\map{p_U}{U'}{U}$
be a finitely presented \etale{} morphism. Then, the pushout $X$ of $j'$ and
$p_U$ exists in the category of quasi-compact and quasi-separated algebraic
stacks. The resulting co-cartesian diagram
$$\xymatrix{U'\ar[r]^{j'}\ar[d]_{p_U} & X'\ar[d]^p\\
U\ar[r]_{j} & X\ar@{}[ul]|\square}$$
is also cartesian, $j$ is a quasi-compact open immersion and $p$ is an
\etale{} and finitely presented neighborhood of $X\setminus U$.
%
\end{theorem}

\begin{theorem}[{\cite[Thm.~D]{rydh_etale-devissage}}]\label{T:etnbhd:devissage}
Let $X$ be a quasi-compact and quasi-separated algebraic stack and let $\catE$
be the $2$-category of finitely presented \etale{} morphisms $Y\to X$. Let
$\catD\subseteq \catE$ be a full subcategory such that
{\renewcommand{\theenumi}{{\upshape{(D\arabic{enumi})}}}
\begin{enumerate}
\item \label{TI:etdev:first} if $Y\in \catD$ and $(Y'\to Y)\in \catE$, then
  $Y'\in \catD$,
\item \label{TI:etdev:second}
  if $Y'\in \catD$ and $Y'\to Y$ is finite, surjective and \etale{}, then
  $Y\in \catD$, and
\item \label{TI:etdev:third}
  if $\map{j}{U}{Y}$ and $\map{f}{Y'}{Y}$ are morphisms in $\catE$ such that
  $j$ is an open immersion, $f$ is an \etale{} neighborhood of $Y\setminus
  U$ and $U,Y'\in \catD$, then $Y\in \catD$.
\end{enumerate}}
Then, if $(Y'\to Y)\in \catE$ is \emph{representable} and surjective and $Y'\in
\catD$, it follows that $Y\in \catD$. In particular, if there exists a
\emph{representable} and surjective morphism $Y\to X$ in $\catE$ with
$Y\in \catD$,
then $\catD=\catE$.
\end{theorem}

Note that the morphisms in $\catE$ are not necessarily representable nor
separated. In Theorem~\pref{T:etnbhd-pushout}, even if $X'$ and $U$ have
separated diagonals, the pushout $X$ need not unless $p_U$ is
representable. We are thus naturally led to include algebraic stacks with
non-separated diagonals.
\end{section}


\begin{section}{Approximation of modules and algebras}\label{S:completeness}
In this section, we prove Theorem~\tref{T:COMPLETENESS}, that is, that every
stack of approximation type is \emph{pseudo-noetherian}. It is known
that noetherian stacks are pseudo-noetherian~\cite[Prop.~15.4]{laumon} but
the non-noetherian case requires completely different methods.
First we prove that stacks of strict approximation type are pseudo-noetherian
and then we deduce the theorem for stacks of approximation type by \etale{}
\devissage{}.

\begin{xpar}\label{X:categories}
Let $X$ be a quasi-compact and quasi-separated algebraic stack. Let
$\catC$ be one of the following categories.
\begin{enumerate}
\item The category of quasi-coherent $\sO_X$-modules.
\item The category of quasi-coherent $\sO_X$-algebras.
\item The category of integral quasi-coherent $\sO_X$-algebras.
\end{enumerate}
Here integral has the meaning as in ``integral closure'' not as in
``integral domain''.
If $U\subseteq X$ is an open substack we denote the corresponding category
of $\sO_U$-modules by $\catC_U$. Consider the following statements.

\vspace{5 mm}
\noindent\emph{Completeness}
{\renewcommand{\theenumi}{{\upshape{(C\arabic{enumi})}}}
\begin{enumerate}
\item Every object in $\catC$ is the direct limit of its
subobjects of finite type.\label{I:C1}
\item Every object in $\catC$ is a filtered direct limit of finitely
presented objects in $\catC$.\label{I:C2}
\end{enumerate}}

\noindent\emph{Presentation} ---
Let $\sF$ be an object in $\catC$ of finite type.
{\renewcommand{\theenumi}{{\upshape{(P\arabic{enumi})}}}
\begin{enumerate}
\item There exists a finitely presented object $\sP$ and a surjection
$\sP\surj \sF$.\label{I:P1}
\item There is a filtered direct system of finitely presented
objects in $\catC$ with surjective bonding maps and limit $\sF$.\label{I:P2}
\end{enumerate}}

\noindent\emph{Extension} ---
Let $U\subseteq X$ be a quasi-compact open substack.
{\renewcommand{\theenumi}{{\upshape{(E\arabic{enumi})}}}
\begin{enumerate}
\item If $\sF_U\in\catC_U$ is of finite type (resp.\ finite presentation), then
there exists an object $\sF\in \catC$ of finite type (resp.\ finite
presentation) such that $\sF|_U=\sF_U$.\label{I:E1}
\item If $\sG\in\catC$ is arbitrary and $\sF_U\in \catC_U$ is of finite type
(resp.\ finite presentation), together with a homomorphism
$\map{u}{\sF_U}{\sG|_U}$, then there exists an object $\sF\in\catC$
of finite type (resp.\ finite presentation) and a homomorphism
$\map{v}{\sF}{\sG}$ extending $\sF_U$ and $u$. To be precise, there exists an
isomorphism $\map{\theta}{\sF|_U}{\sF_U}$ such that
$v|_U=u\circ\theta$.\label{I:E2}
\end{enumerate}}

Note that \ref{I:C1} follows from \ref{I:C2}, that \ref{I:P1} follows from
\ref{I:P2} and that \ref{I:E1} is a special case of \ref{I:E2} (take $\sG=0$).
Also, given $\sF_U$, $\sG$ and $u$ as in \ref{I:E2}, there is a \emph{universal}
extension $\map{v}{\sF}{\sG}$ of $u$ if we drop the condition that $\sF$ is of
finite type. Indeed, if $\map{j}{U}{X}$ is the inclusion morphism, then the
universal solution is $\sF=\sG\times_{j_*j^*\sG} j_*\sF_U$ together with the
projection onto the first factor. If $\catC$ is the category of integral
$\sO_X$-algebras, then the universal solution is the integral closure of
$\sO_X$ in
$\sF$ (as a subring of $\sF$). If $u$ is injective then so is $v$.
\end{xpar}

\begin{definition}
Let $X$ be a quasi-compact and quasi-separated stack. We say that $X$ has the
\emph{completeness property} if the six properties
\ref{I:C1}, \ref{I:C2}, \ref{I:P1}, \ref{I:P2}, \ref{I:E1} and~\ref{I:E2} hold
for $X$ and the categories of quasi-coherent $\sO_X$-modules, $\sO_X$-algebras
and integral $\sO_X$-algebras.
\end{definition}

In the introduction, the completeness property only entailed~\ref{I:C2} for the
category of quasi-coherent $\sO_X$-modules but, as we will see in
Lemma~\pref{L:enough-to-show-P1+C1-or-C2*}, the two definitions are equivalent.
Note that if $X$ has the completeness property, then so has $U\subseteq X$
for any quasi-compact open substack. We also introduce the following auxiliary
condition.
{\renewcommand{\theenumi}{{\upshape{(C2*)}}}
\begin{enumerate}
\item For every object $\sF\in \catC$, there is a filtered direct system
of finitely presented objects $\sF_\lambda\in \catC$ and a surjection
$\varinjlim_{\lambda} \sF_\lambda\surj \sF$.\label{I:C2*}
\end{enumerate}}

\begin{lemma}\label{L:enough-to-show-P1+C1-or-C2*}
Let $X$ be a quasi-compact and quasi-separated stack. Let $\catC$ be one of the
three categories in~\pref{X:categories}. Then the following conditions are
equivalent.
\begin{enumerate}
\item \ref{I:C2} holds for $X$ and $\catC$.
\item \ref{I:C2*} holds for $X$ and $\catC$.
\item \ref{I:C1} and \ref{I:P1} hold for $X$ and $\catC$.
\item $X$ has all six properties for $\catC$.
\end{enumerate}
Moreover, if $X$ has property \ref{I:C2} for the category of quasi-coherent
modules, then $X$ has the completeness property.
\end{lemma}
\begin{proof}
Clearly \ref{I:C2}$\implies$\ref{I:C2*}$\implies$\ref{I:C1}$+$\ref{I:P1} (to
see the second implication, pass to a presentation of $X$ by an affine scheme).
As we noted above, \ref{I:E1} is a special case of \ref{I:E2}. We will show
three other implications from which the first part of the lemma follows.

\ref{I:C1}$+$\ref{I:P1}$\implies$\ref{I:P2}: Let $\sF$ be of finite type
and let $\sP\surj \sF$ be a surjection with $\sP$ finitely presented. Let
$\sK\subseteq \sP$ be the kernel. Then $\sK$ is the limit of its submodules
(or subideals) $\sK_\lambda$ of finite type and it follows that
$\sF=\varinjlim_\lambda
\sP/\sK_\lambda$ is a limit of finitely presented objects.

\ref{I:C1}$+$\ref{I:P1}$\implies$\ref{I:C2}: Let $\sF$ be a quasi-coherent
sheaf. Then $\sF=\varinjlim_{\lambda\in L} \sF_\lambda$ where $\sF_\lambda$ is
of finite type. Let $\sP_\lambda$ be a finitely presented object with a
surjection onto $\sF_\lambda$. For a finite subset $J\subseteq L$ we let
$\sP_J$ be the coproduct of $\{\sP_\lambda\}_{\lambda\in J}$ in $\catC$ and let
$\sK_J$ be the kernel of the induced homomorphism $\sP_J\to \sF$. Consider the
set of pairs $\alpha=(J,\sR_J)$ where $J\subseteq L$ is a finite subset and
$\sR_J\subseteq \sK_J$ is a submodule (or subideal) of finite type. Then
$\sF=\varinjlim_\alpha \sP_J/\sR_J$ is a filtered direct limit of finitely
presented objects (cf.\ proof of \cite[Cor.~6.9.12]{egaI_NE}).

\ref{I:C1}+\ref{I:P1}$\implies$\ref{I:E2}: Let $\sF_U$ be a quasi-coherent sheaf
on $U$ of finite type (resp.\ of finite presentation) and let
$\map{u}{\sF_U}{\sG|_U}$ be a homomorphism as in \ref{I:E2}. Let
$\map{v}{\sF}{\sG}$ be the universal extension. Then as $\sF_U=\sF|_U$ is of
finite type, it follows from \ref{I:C1} that there exists a subsheaf
$\sF'\subseteq \sF$ of finite type which restricts to $\sF_U$. If $\sF_U$ is
finitely presented, write $\sF'=\sP/\sK$ with $\sP$ of finite presentation.
Then $\sK|_U$ is of finite type and hence by \ref{I:C1} there exists a
submodule (or subideal) $\sK'\subseteq \sK$ of finite type which restricts
to $\sK|_U$. The homomorphism $\sP/\sK'\surj \sF'\inj \sF\to \sG$ is the
requested extension of $u$.

To prove the last statement, assume that $X$ has property \ref{I:C2} for the
category of quasi-coherent sheaves of modules. Let $\sA$ be a sheaf of algebras
on $X$. Considering $\sA$ as an $\sO_X$-module, we can then write
$\sA=\varinjlim_\lambda \sF_\lambda$ as a filtered direct limit of finitely
presented modules. If we then let $\sA_\lambda$ be the symmetric algebra of
$\sF_\lambda$, we have a surjection $\varinjlim_\lambda \sA_\lambda\surj \sA$
as in~\ref{I:C2*}.
%
This settles the completeness property for the category of algebras.

If $\sA$ is an integral algebra, then it is a direct limit of its integral
subalgebras since any subalgebra of an integral algebra is integral. This
settles \ref{I:C1} for the category of integral algebras. If $\sA$ is of finite
type then we can, using \ref{I:P2} for the category of algebras, write $\sA$ as
a filtered direct limit of finitely presented algebras $\sB_\lambda$ with
surjective bonding maps. Then $\sB_\lambda$ is integral for sufficiently large
$\lambda$. Indeed, this is easily verified after passing to an affine
presentation. This shows \ref{I:P1} for the category of integral algebras.
\end{proof}

\begin{remark}\label{R:C2-with-open}
Let $X$ be a stack with the completeness property and let $\sF$ be a sheaf in
one of the categories referred to above. If $U$ is a quasi-compact open
substack such
that $\sF|_U$ is of finite type (resp.\ of finite presentation), then $\sF$ is
the direct limit of its finite type subsheaves (resp.\ a filtered direct limit
of finitely presented sheaves) $\sF_\lambda$ such that $\sF_\lambda|_U\to
\sF|_U$ is an isomorphism. Indeed, this follows by a similar argument as in
the proof that \ref{I:C1}+\ref{I:P1} implies \ref{I:C2} above.
\end{remark}

\begin{remark}[Generators]
Recall that a subset $\catG\subseteq \catC$ is \emph{generating} if a morphism
$\map{f}{\sF}{\sG}$ in $\catC$ is zero if and only if $f\circ p=0$ for every
morphism $\map{p}{\sP}{\sF}$ with $\sP\in \catG$. We introduce the following
conditions for the categories in~\pref{X:categories}.
{\renewcommand{\theenumi}{{\upshape{(G\arabic{enumi})}}}
\begin{enumerate}
\item The objects of finite type generate $\catC$.\label{I:G1}
\item The objects of finite presentation generate $\catC$.\label{I:G2}
\end{enumerate}}
It is straight-forward to deduce that \ref{I:C1}$\iff$\ref{I:G1} and that
\ref{I:C2}$\implies$\ref{I:G2}$\implies$\ref{I:G1}$+$\ref{I:P1} so \ref{I:G2}
for the category of quasi-coherent modules is equivalent to the completeness
property. Moreover, the compact objects in the categories of~\pref{X:categories}
are exactly the finitely presented objects. Thus, condition~\ref{I:G2}, or
equivalently~\ref{I:C2}, holds
for $\catC$ if and only if $\catC$ is compactly generated.
\end{remark}

\begin{proposition}\label{P:completeness-property-stable-under-qaffine}
Let $X$ be a quasi-compact and quasi-separated stack with the completeness
property and let $\map{f}{X'}{X}$ be quasi-affine. Then $X'$ has the
completeness property.
\end{proposition}
\begin{proof}
Let $\sF$ be a quasi-coherent $\sO_{X'}$-module. Since $X$ has the completeness
property, we can write $f_*\sF=\varinjlim_\lambda \sG_\lambda$ as a filtered
direct limit of finitely presented $\sO_X$-modules. As $f$ is quasi-affine, the
counit homomorphism $f^*f_*\sF\to \sF$ is surjective. We thus obtain a
surjection $\varinjlim_\lambda f^*\sG_\lambda\surj \sF$ so
condition~\ref{I:C2*} holds for $X'$ and the stack $X'$ has the completeness
property by Lemma~\pref{L:enough-to-show-P1+C1-or-C2*}.
\end{proof}

For an affine scheme properties \ref{I:C1} and \ref{I:P1} are straight-forward
and hence any quasi-affine scheme has the completeness property. More
generally, the completeness property has been shown for quasi-compact and
quasi-separated schemes by Grothendieck~\cite[\S6.9]{egaI_NE}, for
\emph{noetherian} algebraic spaces by Knutson~\cite[Thm.~III.1.1,
Cor.~III.1.2]{knutson_alg_spaces} and for
\emph{noetherian} algebraic stacks by Laumon and
Moret-Bailly~\cite[Prop.~15.4]{laumon}.

\begin{definition}\label{D:pseudo-noetherian}
An algebraic stack $X$ is \emph{pseudo-noetherian} if it is quasi-compact,
quasi-separated and $X'$ has the completeness property for any finitely
presented morphism $X'\to X$ of algebraic stacks.
\end{definition}

In particular, any noetherian stack is pseudo-noetherian.

\begin{proposition}\label{P:pseudo-noetherian-stable-under-strict-app-type}
Let $S$ be a pseudo-noetherian stack and let $X\to S$ be of strict approximation
type.
\begin{enumerate}
\item There is a factorization of $X\to S$ into an affine morphism $X\to X_0$
followed by a finitely presented morphism $X_0\to S$.
\item $X$ is pseudo-noetherian.
\end{enumerate}
In particular, stacks of strict approximation type (e.g., quasi-affine schemes)
are pseudo-noetherian.
\end{proposition}
\begin{proof}
It is enough to prove (i) when there is a factorization $X\to Y\to S$
such that $X\to Y$ is finitely presented and $Y\to S$ is affine.
Since $S$ has the completeness
property, we can write $Y=\varprojlim_\lambda Y_\lambda$ where $Y_\lambda\to
S$ are finitely presented and affine morphisms. For sufficiently large
$\lambda$, there is a morphism $X_\lambda\to Y_\lambda$ of finite
presentation between algebraic stacks such that $X=X_\lambda\times_{Y_\lambda}
Y$. This follows from Proposition~\pref{P:std-limits}. The requested
factorization is obtained by letting $X_0=X_\lambda$ since $X\to X_\lambda$ is
affine and $X_\lambda\to Y_\lambda\to S$ is of finite presentation.

(ii) Let $X'\to X$ be of finite presentation. We have to show that $X'$ has the
completeness property. As $X'\to X\to S$ is of strict approximation type we
have a factorization $X'\to X'_0\to S$ consisting of an affine morphism
followed by a finitely presented morphism. It follows from
Proposition~\pref{P:completeness-property-stable-under-qaffine} that $X'$ has
the completeness property since $X'_0$ has the completeness property by
definition.

The last statement follows from the fact that $\Spec \Z$ is pseudo-noetherian.
\end{proof}

The proof of the following result is inspired by a similar argument due to
P.\ Gross.

\begin{lemma}\label{L:completeness:fin-flat}
Let $X$ be a quasi-compact and quasi-separated stack and let $\map{p}{X'}{X}$
be finite and faithfully flat of finite presentation. If $X'$ has
the completeness property, then so has $X$. Thus, $X$ is
pseudo-noetherian if and only if $X'$ is pseudo-noetherian.
\end{lemma}
\begin{proof}
Assume that $X'$ has the completeness property. By
Lemma~\pref{L:enough-to-show-P1+C1-or-C2*} it is enough to show
that~\ref{I:C2*} holds for $X$. Let $\sF$ be a quasi-coherent $\sO_X$-module.
Recall that $\map{p_*}{\QCoh(X')}{\QCoh(X)}$ has a right adjoint
$\map{p^{!}}{\QCoh(X)}{\QCoh(X')}$ defined by
$$p^{!}\sF=p^{-1}\sHom_{\sO_X}(p_*\sO_{X'},\sF)\otimes_{p^{-1}p_*\sO_{X'}}
\sO_{X'}.$$
Moreover, the counit homomorphism $p_*p^{!}\sF\to \sF$ is surjective since
$p$ is faithfully flat. Write $p^{!}\sF$ as a filtered direct limit of finitely
presented $\sO_{X'}$-modules $\sG_\lambda$. Then $\varinjlim_{\lambda}
p_*\sG_\lambda=p_*p^{!}\sF\to \sF$ is surjective and \ref{I:C2*} holds for~$X$.
\end{proof}

For our immediate purposes we only need
Lemma~\pref{L:completeness:fin-flat} for finite \etale{} coverings in
which case $p^{!}=p^{-1}$.

We now come to the step that involves \etale{} neighborhoods.
The method is inspired by Raynaud and
Gruson's proof of property~\ref{I:C1} for quasi-compact and quasi-separated
algebraic spaces~\cite[Prop.~5.7.8]{raynaud-gruson}.

\begin{lemma}\label{L:completeness:et-nbhd}
Let $X$ be a quasi-compact and quasi-separated stack, let $U\subseteq X$ be a
quasi-compact open substack and let $\map{p}{X'}{X}$ be a finitely presented
\etale{} neighborhood of $X\setminus U$. If $X'$ and $U$ have the completeness
property, then so has $X$. In particular, if $X'$ and $U$ are
pseudo-noetherian, then so is $X$.
\end{lemma}
\begin{proof}
Let $U'=p^{-1}(U)$. By Lemma~\pref{L:enough-to-show-P1+C1-or-C2*} it is
enough to show that~\ref{I:C1} and~\ref{I:P1} hold for $X$.

We begin with~\ref{I:C1}. Let $\sF$ be a quasi-coherent sheaf on $X$. As $X'$
has property~\ref{I:C1}, we have that $p^*\sF$ is the direct limit of its
subsheaves of finite type. It is thus enough to show that if $\sG'\subseteq
p^*\sF$ is of finite type, then there exists $\sH\subseteq \sF$ of finite type
such that $\sG'\subseteq p^*\sH$. As $U$ has property~\ref{I:C1}, there is a
subsheaf $\sH_U\subseteq \sF|_U$ of finite type on $U$ such that
$\sG'|_{U'}\subseteq (p|_U)^*\sH_U$.
Let $\overline{\sG'}\subseteq p^*\sF$ be the universal extension of
$(p|_U)^*\sH_U\subseteq p^*\sF|_{U'}$; then $\sG'\subseteq
\overline{\sG'}\subseteq p^*\sF$. As $\sG'$ and
$\overline{\sG'}|_{U'}=(p|_U)^*\sH_U$ are of finite type, it follows from
property~\ref{I:C1} on $X'$ that there exists a subsheaf $\sH'\subseteq
\overline{\sG'}$ of finite type, containing $\sG'$ and restricting to
$(p|_U)^*\sH_U$ over $U'$. By Theorem~\pref{T:etnbhd-descent:QCoh}, there is a
subsheaf $\sH\subseteq \sF$ of finite type with isomorphisms $\sH|_U\iso \sH_U$
and $p^*\sH\iso \sH'$. This settles property~\ref{I:C1}.

We continue with property~\ref{I:P1}. Let $\sF$ be a quasi-coherent sheaf on
$X$ of finite type. As $U$ has property~\ref{I:P2}, we can write $\sF|_U$ as a
direct limit $\varinjlim \sP_{U,\lambda}$ of finitely presented sheaves on $U$
with surjective bonding maps.
As $X'$ has property~\ref{I:P1}, there is a finitely presented sheaf $\sQ'$ on
$X'$ and a surjection $\sQ'\surj p^*\sF$. For sufficiently large $\lambda$
we have a factorization
$$\sQ'|_{U'}\to (p|_U)^*\sP_{U,\lambda}\surj
p^*\sF|_{U'},$$
cf.\ proof of~\cite[Thm.~8.5.2]{egaIV}.
Moreover, after increasing $\lambda$ we may assume that the homomorphism
$\sQ'|_{U'}\to (p|_U)^*\sP_{U,\lambda}$ is surjective.

Let $\sK'=\ker(\sQ'\surj p^*\sF)$ and $\sN'_{U'}=\ker(\sQ'|_{U'}\surj
(p|_U)^*\sP_{U,\lambda})\subseteq \sK'|_{U'}$. As $\sN'_{U'}$ is of finite
type, there exists, by~\ref{I:E2}, a subsheaf $\sN'\subseteq \sK'$ of finite
type such that $\sN'|_{U'}=\sN'_{U'}$. Let $\sP'=\sQ'/\sN'$. This is a finitely
presented sheaf on $X'$ with a surjection onto $p^*\sF$ such that
$\sP'|_{U'}=(p|_U)^*\sP_{U,\lambda}$. By Theorem~\pref{T:etnbhd-descent:QCoh},
there is a finitely presented $\sO_X$-module $\sP$ and a surjection
$\sP\surj \sF$ which restricts to $\sP_{U,\lambda}\surj \sF|_U$ and
$\sP'\surj p^*\sF$ over $U$ and $X'$.
\end{proof}

\begin{proposition}\label{P:pseudo-noeth:et}
Let $X$ be an algebraic stack and let $\map{p}{X'}{X}$
be \etale{}, representable, surjective and of finite presentation. Then
$X$ is pseudo-noetherian if and only if $X'$ is pseudo-noetherian.
\end{proposition}
\begin{proof}
The condition is necessary by definition. To show that it is sufficient, let
$\catD\subseteq \catE=\Stack_{\fp,\metale/X}$ be the full subcategory with
objects \etale{} and finitely presented morphisms $Y\to X$ such that $Y$ is
pseudo-noetherian. By the definition of a pseudo-noetherian stack, the category
$\catD$ satisfies condition~\ref{TI:etdev:first} of
Theorem~\pref{T:etnbhd:devissage}. That $\catD$ satisfies
conditions~\ref{TI:etdev:second} and~\ref{TI:etdev:third} follows from
Lemmas~\pref{L:completeness:fin-flat} and~\pref{L:completeness:et-nbhd}.
Since $X'\in \catD$ we conclude from Theorem~\pref{T:etnbhd:devissage}
that $X\in \catD$, i.e., that $X$ is pseudo-noetherian.
\end{proof}


\begin{corollary}
Every quasi-compact and quasi-separated Deligne--Mumford stack is
pseudo-noetherian.
\end{corollary}

Finally, we prove Theorem~\tref{T:COMPLETENESS}, that is, that every stack of
approximation type is pseudo-noetherian. We give the following slightly
stronger version.

\begin{theorem}\label{T:pseudo-noetherian-stable-under-app-type}
Let $X$ be a pseudo-noetherian stack and let $X'\to X$ be of approximation
type. Then $X'$ is pseudo-noetherian. In particular,
stacks of approximation type are pseudo-noetherian.
\end{theorem}
\begin{proof}
As $X'\to X$ is of approximation type, there is by definition a surjective
representable
and finitely presented \etale{} morphism $X''\to X'$ such that $X''\to X$
is of strict approximation type. By
Proposition~\pref{P:pseudo-noetherian-stable-under-strict-app-type} we have
that $X''$ is pseudo-noetherian and it follows that $X'$ is pseudo-noetherian
by Proposition~\pref{P:pseudo-noeth:et}.
The last statement follows from the fact that $\Spec \Z$ is pseudo-noetherian.
\end{proof}

\end{section}


\begin{section}{Finite coverings of stacks}\label{S:finite-coverings}
In this section, we prove Theorem~\tref{T:FINITE-COVERINGS}, that is, that
every quasi-compact stack $X$ with quasi-finite and separated diagonal
admits a finite surjective morphism of finite presentation from a
\emph{scheme} $Z$ that is flat over a dense quasi-compact open substack
$U\subseteq X$. Furthermore, if $X$ is Deligne--Mumford, then there is such a
$Z$ which is \etale{} over $U$.

\begin{lemma}[Variant of Zariski's Main Theorem]\label{L:qfin+int-closed=open}
Let $\map{f}{X}{Y}$ be a representable quasi-finite and separated morphism
of algebraic stacks such that $\sO_Y\to f_*\sO_X$ is injective and
integrally closed. Then $f$ is an open immersion.
\end{lemma}
\begin{proof}
As the question is fppf-local on $Y$ and the integral closure commutes with
smooth base change~\cite[Prop.~16.2]{laumon}, we can assume that $Y$ is an
affine scheme. The result then follows from Zariski's Main
Theorem~\cite[18.12.13 and~8.12.3]{egaIV}.
\end{proof}

\begin{proof}[Proof of Theorem~\tref{T:FINITE-COVERINGS}]
%
Let $\map{\pi}{X'}{X}$ be a separated and quasi-finite flat
(resp.\ separated and quasi-compact \etale{}) presentation by a scheme $X'$,
as exists
by~\cite[Thm.~7.1]{rydh_etale-devissage}.
%
The separable fiber rank of $\pi$ is constructible and lower
semi-continuous~\cite[Cor.~9.7.9, Prop.~15.5.9]{egaIV}.
%
There is thus a quasi-compact open dense substack $U\subseteq X$ such that the
separable rank is locally constant on $U$. Let $U=U_1\amalg U_2\amalg\dots
\amalg U_n$ be the decomposition into open and closed substacks such that
$\pi$ has constant separable fiber rank $d$ over $U_d$. The theorem follows if
we construct a scheme $Z_d$ and a finite and finitely presented morphism
$\map{q}{Z_d}{X}$ such that $q|_{U_d}$ is flat (resp.\ \etale{}) and surjective
and such that $q^{-1}(U_k)=\emptyset$ for $k\neq d$.

To simplify notation, set $U=U_d$ and let $U'=\pi^{-1}(U)$. Let
$(U'/U)^d=U'\times_U U'\times_U \dots\times_U U'$ and let
$V=\SEC_d(U'/U)\subseteq (U'/U)^d$ be the open subscheme given by the
complement of the union of all diagonals.
The structure morphism
$V\to U$ is quasi-finite, flat, finitely presented and separated with fibers of
separable rank $d!$. It follows that $V\to U$ is
finite~\cite[Prop.~15.5.9]{egaIV}. If $\pi$ is \etale{}, then $V\to U$ is also
\etale{}. Let $\map{p}{W}{X}$ be the normalization of $X$ in $V$. Then $p$ is
surjective and integral, the restriction $\map{p|_U}{W|_U\iso V}{U}$ is flat
and of
finite presentation (resp.\ \etale{}) and $p(W)=\overline{U}$. We
will now show that $W$ is a scheme.

Let $W'=W\times_X X'$ and $V'=V\times_U U'$; then $W'$ is a scheme. We
have $d$ sections $\map{s_i}{V}{V'}$ such that $\bigcup_i |s_i(V)|=|V'|$ as
\emph{sets}. Let $Y_i=\overline{s_i(V)}$ be the scheme-theoretic closure of
the section $s_i$ in
$W'$. Then $|W'|=\bigcup_i |Y_i|$ since $\map{\pi_W}{W'}{W}$ is flat. As
$W$ is integrally closed in $V$ and $s_i(V)\iso V$ is schematically dense in
$Y_i$, we have that $\sO_{W}\to (\pi_W)_*\sO_{Y_i}$ is injective and integrally
closed. Thus, $\map{(\pi_W)|_{Y_i}}{Y_i}{W}$ is an open immersion by
Lemma~\pref{L:qfin+int-closed=open}. In particular, we have that $W=\bigcup_i
Y_i$ is a \emph{scheme}. The morphism $\map{p}{W}{X}$ is integral and
surjective and the restriction $\map{p|_U}{V}{U}$ is finite, flat and
finitely presented (resp.\ \etale{}). The final step is to approximate $p$ with
a finitely presented morphism $\map{p_\lambda}{W_\lambda}{X}$ such that
$W_\lambda$ is a scheme.

Let $\sA=p_*\sO_W$. We now use that $X$ has the completeness property
(Theorem~\tref{T:COMPLETENESS}) and write $\sA$ as a direct limit of finite and
finitely presented algebras $\sA_\lambda$ such that $(\sA_\lambda)|_U=\sA|_U$
and $\sA_\lambda|_{U_k}=0$ for all $k\neq d$, cf.\ Remark~\pref{R:C2-with-open}.
Let $W_\lambda=\Spec_X(\sA_\lambda)$; then $V=U\times_X W_\lambda$.  Since
$Y_i\to W$ is an open immersion, we have that $Y_i\inj W'$ is a finitely
presented closed immersion. By standard limit methods, there exists $\lambda$
and a finitely presented closed immersion
$(Y_i)_\lambda\inj W'_\lambda=W_\lambda\times_X X'$ of
schemes which pull-backs to $Y_i\inj W'$. After increasing $\lambda$ we can
further assume, by Proposition~\pref{P:std-limits:properties}, that the
composition $(Y_i)_\lambda\inj W'_\lambda\to W_\lambda$ is an open immersion and
that $\coprod_i (Y_i)_\lambda\to W_\lambda$ is an open covering. In
particular, we have that $W_\lambda$ is a scheme.
\end{proof}

\end{section}


\begin{section}{Properties stable under approximation}
\label{S:approximation-properties}
Let $X=\varprojlim_\lambda X_\lambda$ be an inverse limit of
\emph{finitely presented}
stacks over $S$. In this section, we prove
Theorem~\tref{T:APPROXIMATION-PROPERTIES} for the system $(X_\lambda\to S)$,
that is, we prove that if $X\to S$ has a
certain property~$P$, then so has $X_\lambda\to S$ for sufficiently large
$\lambda$. This result is more elementary than the previous theorems and
essentially independent of these. In fact, we only use the previous results
when $P$ is either ``separated'' or ``proper with finite diagonal''. Most of
the properties are deduced by passing to the diagonal via the following lemma.

\begin{lemma}\label{L:diag-inv-system}
Let $S$ be a quasi-compact algebraic stack. Let $X=\varprojlim_{\lambda\in
L} X_\lambda$ be a limit of quasi-compact and quasi-separated morphisms
(resp.\ quasi-separated morphisms of finite type) with affine bonding maps.
\begin{enumerate}
\item The morphisms $\map{g_\lambda}{X\times_{X_\lambda} X}{X\times_S X}$ are
representable and of finite type (resp.\ of finite presentation). For every
$\mu\geq \lambda$ the
morphism $\map{g_{\mu\lambda}}{X\times_{X_\mu} X}{X\times_{X_\lambda} X}$ is a
closed immersion.
\item The inverse system $\{g_\lambda\}_{\lambda\in L}$ has limit
$\map{\Delta_{X/S}}{X}{X\times_S X}$.
\end{enumerate}
\end{lemma}
\begin{proof}
The morphism $X\times_{X_\lambda} X\to X\times_S X$ is a pull-back of the
diagonal $\Delta_{X_\lambda/S}$ and hence representable and of finite
type (resp.\ finite presentation).
The bonding map $g_{\mu\lambda}$ is a pull-back of the diagonal
of the affine morphism $X_\mu\to X_\lambda$.

Let $L$ be the limit stack of the inverse system $\{g_\lambda\}$. By the
universal property of the inverse limit $L$, the diagonals $X\inj
X\times_{X_\lambda} X$ factor through $L$ and the resulting map $X\to L$ is a
monomorphism. Similarly, by the universal property of $X$, the two projections
$\map{\pi_1,\pi_2}{L}{X\times_{X_\lambda} X\to X}$ coincide. It follows that
$L\to X\times_{X_\lambda} X$ factors through $\Delta_{X/X_\lambda}$ and hence
that $L=X$.
\end{proof}

\begin{proposition}[{\cite[Prop.~C.6]{thomason-trobaugh}}]
\label{P:rel-affine-limit}
Let $S$ be a quasi-compact algebraic stack and let $X=\varprojlim_{\lambda\in
L} X_\lambda$ be a limit of finitely presented $S$-stacks
with affine bonding maps. If $X\to S$ is
affine (resp.\ quasi-affine), then there is an index $\alpha$ such that
$X_\lambda\to S$ is affine (resp.\ quasi-affine) for every $\lambda\geq
\alpha$.
\end{proposition}
\begin{proof}
We will first show the proposition under the assumption that the morphisms
$X_\lambda\to S$ have affine double diagonals (e.g., separated diagonals).

The question is local on $S$ in the fppf topology, so we can assume that $S$ is
an affine scheme. Let $\overline{X}=\Spec(\Gamma(\sO_X))$ be the affine hull of
$X$; recall that $X\to \overline{X}$ is a quasi-compact open immersion
($X=\overline{X}$ if $X$ is affine). Since $S$ is affine, we can write
$\overline{X}$ as an inverse limit $\varprojlim_{\mu\in M}
\overline{X}'_\mu$ of finitely presented affine $S$-schemes. By
Remark~\pref{R:std-limits:qc-open}, there is an index $\mu_0$ and an open
quasi-compact subscheme $X'_{\mu_0}\subseteq\overline{X}'_{\mu_0}$ with inverse
image $X$ in $\overline{X}$. We let
$X'_\mu=X'_{\mu_0}\times_{\overline{X}'_{\mu_0}} \overline{X}'_{\mu}$ for all
$\mu\geq \mu_0$. Then $X=\varprojlim_{\mu\in M} X'_\mu$ becomes a limit of
finitely presented affine (resp.\ quasi-affine) $S$-schemes.

Let $\alpha_0\in L$ be an index. By the functorial characterization of finitely
presented morphisms, Proposition~\pref{P:functorial-char-of-loc-fin-pres},
there are indices $\alpha\in L$ and $\beta\in M$ and morphisms
$$X_\alpha \to X'_\beta \to X_{\alpha_0}$$
and after increasing $\alpha$, we can assume that the composition coincides
with the bonding map of the system $(X_\lambda)$ and hence is affine. As
$X'_\beta$ has affine diagonal and $X_{\alpha_0}$ has affine
double diagonal, it follows that $X'_\beta\to X_{\alpha_0}$ has affine diagonal
and that $X_\alpha\to X'_\beta$ is affine. Thus, $X_\lambda$ is affine (resp.\
quasi-affine) for every $\lambda\geq\alpha$.

For general $X_\lambda$, we at least know that the triple diagonal is affine
(it is an isomorphism). Repeating the argument above we conclude that
$X_\lambda$ has affine diagonal for every $\lambda\geq\alpha$ and the
proposition follows from the special case.
\end{proof}

\begin{corollary}
\label{C:scheme-limit}
Let $S$ be a quasi-compact \emph{scheme} and let $X=\varprojlim_{\lambda\in
L} X_\lambda$ be a limit of finitely presented $S$-stacks
with affine bonding maps. If $X$ is a scheme, then there is an index $\alpha$
such that $X_\lambda$ is a scheme for every $\lambda\geq\alpha$.
\end{corollary}
\begin{proof}
The question is Zariski-local on $S$ so we can assume that $S$ is
affine. Choose an open affine covering $X=\bigcup_{i=1}^n U_i$. By
Remark~\pref{R:std-limits:qc-open}, there is an index $\lambda$ and open
subsets $U_{i,\lambda}\subseteq X_\lambda$ such that $U_i=U_{i,\lambda}
\times_{X_\lambda} X$ for all $i$. After increasing $\lambda$ we have that
$U_{i,\lambda}$ is affine by Proposition~\pref{P:rel-affine-limit}. Finally,
after further increasing $\lambda$ we can assume that
$X_\lambda=\bigcup_{i=1}^n U_{i,\lambda}$ and then $X_\lambda$ is a scheme.
\end{proof}

\begin{proposition}\label{P:properties-for-closed-limits}
Let $S$ be a quasi-compact algebraic stack and let $X=\varprojlim_\lambda
X_\lambda$ be an inverse limit of finitely presented $S$-stacks such that
the bonding maps $X_\mu\to X_\lambda$ are closed immersions for every
$\mu\geq \lambda$. If
$X\to S$ has one of the following properties:
\begin{enumerate}
\item a monomorphism,\label{TI:cl:first-fiber}
\item universally injective (i.e., ``radiciel''),
\item representable,
\item unramified,
\item quasi-finite,\label{TI:cl:last-fiber}
\item finite,\label{TI:cl:finite}
\item a closed immersion,\label{TI:cl:closedimm}
\item an immersion;\label{TI:cl:imm}
\end{enumerate}
then there exists $\alpha$ such that $X_\lambda\to S$ has the corresponding
property
for all $\lambda\geq \alpha$.

If, in addition, the $X_\lambda$'s and $X$ are $S$-group spaces such that
$X_\mu\inj X_\lambda$ is a subgroup for every $\mu\geq\lambda$, then the same
conclusion holds for the properties:
\begin{enumerate}\setcounter{enumi}{8}
\item abelian fibers,\label{TI:cl:first-fiber-2}
\item quasi-finite with linearly reductive fibers.\label{TI:cl:last-fiber-2}
\end{enumerate}
\end{proposition}
\begin{proof}
As the properties are local in the fppf topology and $S$ is quasi-compact, we
can assume that $S$ is an affine scheme.
Note that $X\inj X_\lambda$ is a closed immersion for every $\lambda$, so
$X\to S$ is of finite type. It follows from Lemma~\pref{L:fiberwise-conds} that
properties \ref{TI:cl:first-fiber}--\ref{TI:cl:last-fiber} can be checked on
fibers. Let $P$ be one of these five properties or one of the properties
\ref{TI:cl:first-fiber-2}--\ref{TI:cl:last-fiber-2} for group spaces. We let
$U_\lambda\subseteq |S|$ be the set of points $s\in |S|$ such that
$(X_\lambda)_s\to \Spec \kappa(s)$ has property~$P$. Then $U_\lambda\subseteq |S|$ is
constructible by Propositions~\pref{P:constructible-conds}
and~\pref{P:constructible-conds-2}.

As a closed immersion has property~$P$, it follows that $U_\lambda\subseteq
U_\mu$ if $\lambda\leq\mu$. If $s\in |S|$ is any point, then as $X_s\to
\Spec \kappa(s)$ is of finite type, we have that $X_s=(X_\lambda)_s$ for
sufficiently large $\lambda$. It thus follows that $|S|=\bigcup U_\lambda$. As
the constructible topology is quasi-compact, it follows that $U_\lambda=|S|$ for
sufficiently large $\lambda$. This completes the demonstration of properties
\ref{TI:cl:first-fiber}--\ref{TI:cl:last-fiber} and
\ref{TI:cl:first-fiber-2}--\ref{TI:cl:last-fiber-2}.

Now assume that $X\to S$ is a closed immersion (resp.\ finite). By
Proposition~\pref{P:rel-affine-limit} we can assume that the maps $X_\lambda\to
S$ are affine.
Let $S=\Spec A$, $X_\lambda=\Spec B_\lambda$ and $X=\Spec B$. Choose an
index $\lambda$ and generators $b_1,b_2,\dots,b_n\in B_\lambda$. The image of
$b_i$ in $B$ lifts to $A$ (resp.\ satisfies a monic equation with coefficients
in $A$). If $a_i\in A$ is a lifting, then the images of $a_i$ and $b_i$
coincides in $B_\mu$ (resp.\ the image of $b_i$ in $B_\mu$ satisfies the monic
equation) for some $\mu\geq\lambda$. As $B_\lambda\to B_\mu$ is surjective it
follows that $A\to B_\mu$ is surjective (resp.\ finite).  This settles
properties~\ref{TI:cl:closedimm} and~\ref{TI:cl:finite}.

If $X\to S$ is an immersion, then let $U\subseteq S$ be an open subscheme
containing the image of $X$ such that $X\to U$ is a closed immersion. As $U$
is ind-constructible, it follows that $X_\lambda\to S$ factors through $U$
for sufficiently large~$\lambda$~\cite[Cor.~8.3.4]{egaIV}.
Property~\ref{TI:cl:imm} thus follows from property~\ref{TI:cl:closedimm}.
\end{proof}

\begin{corollary}\label{C:properties-for-diagonal}
Let $S$ be a quasi-compact algebraic stack and let $X=\varprojlim_\lambda
X_\lambda$ be an inverse limit of algebraic $S$-stacks of finite type with
affine bonding maps. If the
diagonal of $X\to S$ has one of the properties:
\begin{enumerate}
\item a monomorphism,\label{TI:pd-first}
\item unramified,
\item quasi-finite,
\item finite,
\item a closed immersion,\label{TI:pd-closedimm}
\item an immersion,\label{TI:pd-imm}
\item affine,
\item quasi-affine,\label{TI:pd-last}
\item separated,\label{TI:pd-separated}
\item locally separated;\label{TI:pd-loc-separated}
\end{enumerate}
then there exists $\alpha$ such that the diagonal of $X_\lambda\to S$ has the
corresponding property for all $\lambda\geq \alpha$. In particular, if $X/S$
has one of
the properties: representable, representable and separated, representable and
locally separated, relatively Deligne--Mumford, etc.; then so has $X_\lambda/S$.

If the inertia of $X\to S$ has one of the properties:
\begin{enumerate}\setcounter{enumi}{10}
\item finite,\label{TI:pd-inertia-first}
\item abelian fibers,
\item quasi-finite with linearly reductive fibers;\label{TI:pd-inertia-last}
\end{enumerate}
then there exists $\alpha$ such that the inertia of $X_\lambda\to S$ has the
corresponding property for all $\lambda\geq \alpha$.
\end{corollary}
\begin{proof}
Let $P$ be one of the properties~\ref{TI:pd-first}--\ref{TI:pd-last}. By
Lemma~\pref{L:diag-inv-system}, the diagonal $X\to X\times_S X$ is the inverse
limit of the finitely presented morphisms $X\times_{X_\lambda} X\to X\times_S
X$ and the bonding maps $X\times_{X_\mu} X\to X\times_{X_\lambda} X$ are
closed immersions. It thus follows from Propositions~\pref{P:rel-affine-limit}
and \pref{P:properties-for-closed-limits} that if the diagonal of $X/S$ has
property~$P$, then so has $X\times_{X_\lambda} X\to X\times_S X$ for
sufficiently large $\lambda$. As $X\times_S X$ is the inverse limit of
$X_\mu\times_S X_\mu$, it follows by standard limit results,
Proposition~\pref{P:std-limits:properties}, that
$X_\mu\times_{X_\lambda} X_\mu\to
X_\mu\times_S X_\mu$ has property~$P$ for sufficiently large
$\mu\geq\lambda$. As the diagonal $X_\mu\to X_\mu\times_{X_\lambda} X_\mu$ is a
closed immersion, it follows that the diagonal of $X_\mu/S$ has
property~$P$. For properties~\ref{TI:pd-separated}
and~\ref{TI:pd-loc-separated} we reason as above, using that we have proven the
Corollary for properties~\ref{TI:pd-closedimm} and~\ref{TI:pd-imm}.

Let $P$ be one of the
properties~\ref{TI:pd-inertia-first}--\ref{TI:pd-inertia-last}. The pull-back
of the inverse system $X\times_{X_\lambda} X\to X\times_S X$ along the diagonal
$\Delta_{X/S}$ gives the inverse system $I_{X_\lambda/S}\times_{X_\lambda} X\to
X$ with inverse limit the inertia $I_{X/S}\to X$. As the bonding maps in the
first system are closed immersions, the bonding maps in the
second system are closed subgroups.
It thus follows from Proposition~\pref{P:properties-for-closed-limits} that if
the inertia $I_{X/S}\to X$ has property $P$ then so has
$I_{X_\lambda/S}\times_{X_\lambda} X\to X$ for all sufficiently large
$\lambda$. By standard limit results,
Proposition~\pref{P:std-limits:properties}, it follows that
$I_{X_\lambda/S}\times_{X_\lambda} X_\mu\to X_\mu$ has property $P$ for all
sufficiently large $\mu\geq\lambda$ and, a fortiori, so has $I_{X_\mu/S}\inj
I_{X_\lambda/S}\times_{X_\lambda} X_\mu\to X_\mu$.
\end{proof}

\begin{corollary}\label{C:rel-proper-limit}
Let $S$ be a quasi-compact algebraic stack and let $X=\varprojlim_\lambda
X_\lambda$ be an inverse limit of finitely presented $S$-stacks such that
$X_\mu\to X_\lambda$ is a \emph{closed immersion} for every $\mu\geq
\lambda$. If $X\to S$ is proper with finite diagonal, then so is
$X_\lambda\to S$ for all sufficiently large $\lambda$.
\end{corollary}
\begin{proof}
The question is fppf-local on $S$ so we can assume that $S$ is affine. By
Corollary~\pref{C:properties-for-diagonal} we can assume that $X_\lambda\to S$
has finite diagonal. Then there exists a scheme $Z_\lambda$ and a finite and
finitely presented surjective morphism $Z_\lambda \to X_\lambda$ by
Theorem~\tref{T:FINITE-COVERINGS}. We let
$Z_\mu=Z_\lambda\times_{X_\lambda} X_\mu$ for all $\mu>\lambda$ and
$Z=Z_\lambda\times_{X_\lambda} X$. It is then enough to show that $Z_\mu\to S$
is proper for sufficiently large $\mu\geq\lambda$.

Since $Z_\lambda\to S$ is separated and of finite presentation (and $S$ is
affine), there is by Nagata's compactification
theorem~\cite{lutkebohmert_compactification} a proper morphism
$\overline{Z_\lambda}\to S$ and an open immersion $Z_\lambda\subseteq
\overline{Z_\lambda}$. But $Z\to
\overline{Z_\lambda}$ is then a closed immersion so it follows from
Proposition~\pref{P:properties-for-closed-limits} that $Z_\mu\to
\overline{Z_\lambda}$ is a closed immersion for sufficiently large $\mu$.
This shows that $Z_\mu\to S$ is proper.
\end{proof}

\begin{corollary}\label{C:rel-separated-limit}
Let $S$ be a quasi-compact algebraic stack and let $X=\varprojlim_\lambda
X_\lambda$ be an inverse limit of algebraic $S$-stacks of finite type
with affine bonding maps. If $X\to S$ is
separated, then there exists $\alpha$ such that $X_\lambda\to S$ is separated
for every $\lambda\geq\alpha$.
\end{corollary}
\begin{proof}
Reason as in the proof of Corollary~\pref{C:properties-for-diagonal} using
Corollary~\pref{C:rel-proper-limit}.
\end{proof}

We have now proved Theorem~\tref{T:APPROXIMATION-PROPERTIES}
under the additional assumption that every morphism $X_\lambda\to S$ is of
finite presentation and that $S$ is a scheme in~\ref{TI:C-scheme}.
Indeed, this is
Propositions~\pref{P:rel-affine-limit}, \pref{P:properties-for-closed-limits}
and Corollaries~\pref{C:scheme-limit}, \pref{C:properties-for-diagonal},
\pref{C:rel-proper-limit}
and~\pref{C:rel-separated-limit}. In the next section, we will deduce
Theorem~\tref{T:APPROXIMATION-PROPERTIES}
from Theorem~\tref{T:APPROXIMATION} and the following result.

\begin{lemma}\label{L:limit-props-over-another-base}
Let $X\to S$ be a morphism of stacks. Let
$$\{\map{u_\lambda}{X_\lambda}{S}\}_{\lambda\in L}\quad\text{ and }\quad
\{\map{v_\mu}{X'_\mu}{S}\}_{\mu\in M}$$
be two inverse systems with limit $X\to S$ and bonding maps that are affine
(resp.\ closed immersions). Assume that $u_\lambda$ is quasi-compact and
quasi-separated for all $\lambda$ and that $v_\mu$ is of finite presentation
for all $\mu$. Further, assume that there exists $\alpha\in
L$ and a factorization of $u_\alpha$ into an affine morphism (resp.\ closed
immersion) $X_\alpha\to X_0$ followed by a morphism $\map{u_0}{X_0}{S}$ of
finite presentation.

Let $P$ be a property of morphisms of stacks that is stable under composition
with affine morphisms (resp.\ closed immersions). If $v_\mu$ has property $P$
for sufficiently large $\mu$, then so has $u_\lambda$ for sufficiently large
$\lambda$.
\end{lemma}
\begin{proof}
Since $u_0$ is of finite presentation, there is for sufficiently
large $\mu$ a factorization $X\to X'_\mu\to X_0\to S$ by
Proposition~\pref{P:functorial-char-of-loc-fin-pres}. As
$X\to X_\lambda\to X_0$ is affine and $X'_\mu\to X_0$ is of finite
presentation, we can, by
Proposition~\pref{P:rel-affine-limit}, assume that $X'_\mu\to X_0$
is affine after further increasing~$\mu$.

Note that $X\to X_0$ is the limit of the system $\{X_\lambda\to
X_0\}_\lambda$ and that $X'_\mu\to X_0$ is of finite presentation. Thus, we
can apply Proposition~\pref{P:functorial-char-of-loc-fin-pres} and
obtain, for sufficiently large $\lambda$, a factorization
$$X\to X_\lambda\to X'_\mu\to X_0.$$
Since $X_\lambda\to X_0$ is affine (resp.\ a closed immersion), so is
$X_\lambda\to X'_\mu$ and it follows that $X_\lambda\to S$ has property $P$.
\end{proof}


\end{section}


\begin{section}{Approximation of schemes and stacks}\label{S:approximation}
Recall that any stack of approximation type is pseudo-noetherian
(Theorem~\tref{T:COMPLETENESS}) and that any stack that is affine over a
noetherian stack is pseudo-noetherian
(Proposition~\ref{P:pseudo-noetherian-stable-under-strict-app-type}).
Conversely, it is
possible that every pseudo-noetherian stack is affine over a noetherian
stack. In this section, we prove Theorem~\tref{T:APPROXIMATION} and, as a
consequence, that stacks of approximation type are indeed affine over
noetherian stacks.


\begin{definition}
Let $S$ be a pseudo-noetherian stack and let $X\to S$ be a morphism of stacks.
An \emph{approximation} of $X$ over $S$ is a finitely presented
$S$-stack $X_0$ together with an affine $S$-morphism $X\to
X_0$. We say that $X/S$ can be approximated if there exists an approximation
of $X$ over $S$.
\end{definition}

\begin{remark}
In~\cite[8.13.4]{egaIV}, Grothendieck uses the term \emph{essentially affine}
for morphisms of schemes $X\to S$ that can be approximated.
\end{remark}

Let $S$ be pseudo-noetherian. Then
Proposition~\pref{P:pseudo-noetherian-stable-under-strict-app-type} states
that $X\to S$ has an approximation if and only if $X\to S$ is of strict
approximation type. Moreover, if $X\to S$ has an approximation $X\to X_0\to S$
then $X$ and $X_0$ are pseudo-noetherian.

The following two propositions are analogues of properties
\ref{I:C1}--\ref{I:C2} and \ref{I:P1}--\ref{I:P2} under the assumption that
$X/S$ can be approximated.

\begin{proposition}[Completeness]\label{P:completeness:approx}
Let $X/S$ be an algebraic stack that can be approximated. Then
\begin{enumerate}
\item $X=\varprojlim_\lambda X_\lambda$ such that $X_\lambda\to S$ is
of finite type and $X\to X_\lambda$ is schematically dominant.
\item $X=\varprojlim_\lambda X_\lambda$ such that $X_\lambda\to S$ is
of finite presentation.
\end{enumerate}
\end{proposition}

\begin{proposition}[Presentation]\label{P:presentation:approx}
Let $X/S$ be an algebraic stack of \emph{finite type} that can be
approximated. Then
\begin{enumerate}
\item There exists a finitely presented $S$-stack $X_0$ together with a
closed immersion $X\inj X_0$ over $S$.
\item $X=\varprojlim_\lambda X_\lambda$ such that $X_\lambda\to S$ is
of finite presentation and the bonding map $X_\mu\to X_\lambda$ is a closed
immersion
for every $\mu\geq \lambda$.
\end{enumerate}
\end{proposition}

\begin{proof}[Proofs]
Let $X\to X_0\to S$ be an approximation and apply the completeness properties 
\ref{I:C1}--\ref{I:C2} and \ref{I:P1}--\ref{I:P2} on the affine morphism
$\map{f}{X}{X_0}$ (i.e., on the sheaf of $\sO_{X_0}$-algebras $f_*\sO_X$).
\end{proof}

\begin{remark}\label{R:approximation-of-pairs}
If $X\to S$ has an approximation and $U\subseteq X$ is a quasi-compact open
substack, then by standard limit methods, cf.\
Remark~\pref{R:std-limits:qc-open}, there exists an approximation $X_0\to S$ of
$X$ and a quasi-compact open substack $U_0\subseteq X_0$ such that
$U=U_0\times_{X_0} X$. We say that $(U_0\subseteq X_0)\to S$ is an
approximation of $(U\subseteq X)\to S$.
\end{remark}

If $X\to S$ is affine, then $X$ has a trivial approximation, namely $S$
itself. At first, this hardly appears to be an ``approximation'' but the
crucial point is that we assume that $S$ is pseudo-noetherian. Then the
statement that $X\to S$ can be approximated implies that $X$ is the inverse
limit of finitely presented and affine $S$-schemes.

\begin{proposition}\label{P:approx-for-aff/q-aff}
Let $S$ be a pseudo-noetherian stack. If $X$ is a stack that is affine
(resp.\ quasi-affine) over~$S$, then $X$ can be approximated by a stack that
is affine (resp.\ quasi-affine) over~$S$.
\end{proposition}
\begin{proof}
The affine part is trivial, cf.\ the preceding discussion.
If $\map{f}{X}{S}$ is quasi-affine, let $\overline{X}=\Spec_S(f_*(\sO_X))$ so
that
$X\to \overline{X}$ is a quasi-compact open immersion and $\overline{X}\to S$
is affine. By Remark~\pref{R:approximation-of-pairs}, there is an
approximation
$(X\subseteq \overline{X})\to (X_0\subseteq \overline{X}_0)\to S$. The morphism
$X_0\to S$ is a quasi-affine approximation of $X\to S$.
\end{proof}

The following proposition is an analogue of \ref{I:E2}.

\begin{proposition}[Extension]\label{P:corner-stone}
Let $X\to S$ be a morphism of pseudo-noetherian stacks. Let
$U\subseteq X$ be a quasi-compact open substack and let $U=\varprojlim_\lambda
U_\lambda$ be an inverse limit of finitely presented $S$-stacks. If $X\to S$
can be approximated, then there exists an index $\alpha$ such that for any
$\lambda\geq\alpha$, the approximation $U\to U_\lambda\to S$ extends to an
approximation $(U\subseteq X)\to (U_\lambda\subseteq X_\lambda)\to S$.
\end{proposition}
\begin{proof}
By Remark~\pref{R:approximation-of-pairs}, there is an approximation
$(U_0\subseteq X_0)\to S$ of $(U\subseteq X)\to S$. As $U_0\to S$ is finitely
presented, the morphism $U\to U_0$ lifts to $U_\lambda\to U_0$ for sufficiently
large $\lambda$ by Proposition~\pref{P:functorial-char-of-loc-fin-pres}. This
gives us the cartesian diagram
$$\xymatrix{%
U\ar[r]\ar[d] & U_\lambda\ar[r] & U_0\ar[d]\\
X\ar[rr] && X_0\ar@{}[ull]|\square.\\
}$$
As $U\to U_0$ is affine, we have that $U_\lambda\to U_0$ is affine for
sufficiently large $\lambda$ by Proposition~\pref{P:rel-affine-limit}. By
assumption $X_0$ has the completeness property and we can thus,
using~\ref{I:E2}, extend the diagram above to a cartesian diagram
$$\xymatrix{%
U\ar[r]\ar[d] & U_\lambda\ar[r]\ar[d] & U_0\ar[d]\\
X\ar[r] & X_\lambda\ar[r]\ar@{}[ul]|\square & X_0\ar@{}[ul]|\square\\
}$$
where $X_\lambda\to X_0$ is affine and finitely presented. The pair
$(U_\lambda\subseteq X_\lambda)\to S$ is an approximation of $(U\subseteq
X)\to S$.
\end{proof}

We will now proceed with the \etale{} \devissage{} that is used to show that
morphisms of approximation type can be approximated.

\begin{lemma}\label{L:approximation:etale-nbhd}
Let $X\to S$ be a morphism of pseudo-noetherian stacks. Let
$U\subseteq X$ be a quasi-compact open substack and let $\map{p}{X'}{X}$ be
a finitely presented \etale{} neighborhood of $Z=X\setminus U$.
If $U\to S$ and $X'\to S$ can be approximated, then so can $X\to S$.
\end{lemma}
\begin{proof}
Let $U'=p^{-1}(U)$. Write $U$ as an inverse limit $\varprojlim_\lambda
U_\lambda$. For sufficiently large $\lambda$, there exists an \etale{} morphism
$U'_\lambda\to U_\lambda$ of finite presentation such that
$U'_\lambda\times_{U_\lambda} U=U'$. By Proposition~\pref{P:corner-stone} we
can, for sufficiently large $\lambda$, extend the approximation $U'\to
U'_\lambda\to S$ to an approximation $(U'\subseteq X')\to (U'_\lambda\subseteq
X'_\lambda)\to S$.

By Theorem~\pref{T:etnbhd-is-pushout} we have that $X$ is the pushout of
$U'\subseteq X'$ and $\map{p|_U}{U'}{U}$. Let $X_\lambda$ be the pushout of
$U'_\lambda\subseteq X'_\lambda$ and $U'_\lambda\to U_\lambda$. This pushout
exists by Theorem~\pref{T:etnbhd-pushout} and the morphism
$X'_\lambda\amalg U_\lambda\to X_\lambda$ is \etale{} and surjective.
We have furthermore a $2$-cartesian diagram
$$\xymatrix{X'\amalg U\ar[r]\ar[d] & X'_\lambda\amalg U_\lambda\ar[d]\\
X\ar[r] & X_\lambda\ar@{}[ul]|\square}$$
by~\cite[Prop.~3.2]{rydh_etale-devissage}. It follows that
$X_\lambda\to S$ is of finite presentation and that $X\to X_\lambda$ is affine,
so $X\to X_\lambda\to S$ is an approximation.
\end{proof}

\begin{lemma}\label{L:approximation:etale}
Let $X\to S$ be a morphism of pseudo-noetherian stacks.
Let $\map{p}{X'}{X}$ be a representable \etale{} and surjective morphism
of finite presentation. If $X'\to S$ can be approximated, then so can $X\to S$.
\end{lemma}
\begin{proof}
Let $\catD\subseteq \catE=\Stack_{\fp,\metale/X}$ be the full subcategory with
objects \etale{} and finitely presented morphisms $Y\to X$ such that $Y\to X\to
S$ is of strict approximation type. As $S$ is pseudo-noetherian, we have
that $(Y\to X)\in \catD$ if and only if $Y\to X\to S$ can be approximated
(Proposition~\ref{P:pseudo-noetherian-stable-under-strict-app-type}).

The category $\catD$ satisfies condition~\ref{TI:etdev:first} of
Theorem~\pref{T:etnbhd:devissage} by definition. That $\catD$ satisfies
conditions~\ref{TI:etdev:second} and~\ref{TI:etdev:third} follows
from Proposition~\pref{P:strict-approximation:finite-flat-descent} and
Lemma~\pref{L:approximation:etale-nbhd}. Since $X'\in \catD$, we thus
conclude from Theorem~\pref{T:etnbhd:devissage} that $X\in \catD$, i.e., that
$X\to S$ can be approximated.
\end{proof}

\begin{theorem}\label{T:approximation:app-type}
Let $S$ be a pseudo-noetherian stack and let $\map{f}{X}{S}$ be a morphism
of algebraic stacks. The following are equivalent:
\begin{enumerate}
\item $X\to S$ is of approximation type;
\item $X\to S$ is of strict approximation type; and
\item $X\to S$ has an approximation.
\end{enumerate}
\end{theorem}
\begin{proof}
By definition we have that (iii)$\implies$(ii)$\implies$(i).
That (ii)$\implies$(iii) is
Proposition~\pref{P:pseudo-noetherian-stable-under-strict-app-type} and
that (i)$\implies$(ii) is Lemma~\pref{L:approximation:etale}.
\end{proof}

\begin{proof}[Proof of Theorem~\tref{T:APPROXIMATION}]
By Theorem~\pref{T:approximation:app-type} there exists an approximation $X\to
X_0\to S$. By Propositions~\pref{P:completeness:approx}
and~\pref{P:presentation:approx}, we can thus write $X$ as an inverse limit of
finitely presented morphisms $X_\lambda\to S$. The bonding maps $X_\mu\to
X_\lambda$ are affine (resp.\ closed immersions) for general $X\to S$
(resp.\ for $X\to S$ of finite type).
If $X\to S$ is integral, then, since $S$ is pseudo-noetherian, we can arrange so
that $X_\mu\to X_\lambda$ and $X_\lambda\to S$ are finite by the completeness
property~\ref{I:C2} for the
category of integral algebras.

Let $P$ be one of the properties in \ref{I:prop-affine} or \ref{I:prop-closed}.
If $X\to S$
has $P$ then $X_\lambda\to S$ has property~$P$ for all sufficiently large
$\lambda$ by Theorem~\tref{T:APPROXIMATION-PROPERTIES} (finitely presented
case).
The statement on properties of $X\to \Spec \Z$ is an immediate consequence
of the general form of Theorem~\tref{T:APPROXIMATION-PROPERTIES}, proven below.
\end{proof}

We will now reduce the general case of
Theorem~\tref{T:APPROXIMATION-PROPERTIES} to the finitely presented case,
proved in the previous section. This is done by a somewhat subtle bootstrapping
process via Lemma~\pref{L:limit-props-over-another-base} and the
following three lemmas. The first lemma is
Theorem~\tref{T:APPROXIMATION-PROPERTIES} under the additional assumption that
$X_\lambda\to S$ is of approximation type \emph{fppf-locally over $S$}. The
second lemma is the analogue of Proposition~\pref{P:std-limits:properties} but
for morphisms of finite type. It is revisited in
Theorem~\pref{T:std-limits:properties:finite-type}. The third lemma is
Theorem~\tref{T:APPROXIMATION-PROPERTIES} for diagonal and inertia properties.

\begin{lemma}\label{L:TheoremC-for-approx-type}
Let $S$ be a quasi-compact algebraic stack and let $\{X_\lambda\to S\}$ be an
inverse system of quasi-compact and quasi-separated morphisms
(resp.\ quasi-separated morphisms of finite type) of algebraic stacks with
limit $X\to S$ and bonding maps that are affine (resp.\ closed immersions).
\begin{enumerate}
\item 
Assume that $X_\lambda\to S$ is of approximation type fppf-locally over $S$
and let $P$ be one of the properties in~\ref{I:prop-affine}
(resp.\ \ref{I:prop-closed}). Then, if $X\to S$ has property $P$,
so has $X_\lambda\to S$ for all sufficiently large $\lambda$.
\item
Assume that $S$ is Zariski-locally quasi-separated. Then, if $X$ is a scheme,
so is $X_\lambda$ for all sufficiently large $\lambda$.\label{LI:TCat-scheme}
\end{enumerate}
\end{lemma}
\begin{proof}
The first claim is fppf-local on $S$ so we can assume that $X\to S$ and
$X_\lambda\to S$ are of approximation type. The second claim is Zariski-local
on $S$ so we can assume that $S$ is quasi-separated and then replace $S$
with $\Spec \Z$. As morphisms of schemes clearly are of approximation
type, the second claim reduces to the first claim with
$P$ as the property ``strongly representable''.

By
Theorem~\tref{T:APPROXIMATION}, there are approximations
$X=\varprojlim_\mu X'_\mu$ and $X_\alpha\to X_0\to X$ where $X'_\mu\to S$ and
$X_0\to X$ are of finite presentation and the bonding maps, as well as
$X_\alpha\to X_0$, are affine (resp.\ closed immersions). By the finitely
presented case of Theorem~\tref{T:APPROXIMATION-PROPERTIES}, it follows that
$X'_\mu\to S$ has property $P$ for sufficiently large $\mu$. We conclude
that $X_\lambda\to S$ has property $P$ for sufficiently large $\lambda$ by
Lemma~\pref{L:limit-props-over-another-base}.
\end{proof}

\begin{lemma}\label{L:descent-for-finite-type-if-of-approx-type}
Let $S_0$ be a quasi-compact algebraic stack, let $\map{f_0}{X_0}{S_0}$ be a
quasi-separated morphism of finite type and let $S=\varprojlim_\lambda
S_\lambda$ be an inverse system of stacks that are affine over $S_0$. For every
$\lambda$, let $\map{f_\lambda}{X_\lambda}{S_\lambda}$ (resp.\ $\map{f}{X}{S}$)
be the base change of $f_0$ along $S_\lambda\to S_0$ (resp.\ $S\to S_0$). Let
$P$ be one of the properties of \ref{I:prop-affine} or \ref{I:prop-closed}.
Assume that $f_0$ is of
approximation type, fppf-locally over $S_0$. Then $f$ has $P$ if and
only if $f_\lambda$ has $P$ for all sufficiently large $\lambda$.
\end{lemma}
\begin{proof}
The question is fppf-local on $S_0$ so we can assume that $S_0$ is affine and
that $f_0$ is of approximation type.
By Theorem~\tref{T:APPROXIMATION}, we can write $X_0=\varprojlim_\mu X_0^\mu$
as an inverse limit of finitely presented morphisms $X_0^\mu\to S_0$ with
bonding maps that are closed immersions. Then $X=\varprojlim_\mu X^\mu$ where
$X^\mu=X_0^\mu\times_{S_0} S$. By the finitely presented case of
Theorem~\tref{T:APPROXIMATION-PROPERTIES}, it follows that $X^\mu\to S$ has
property $P$ for sufficiently large $\mu$. We then apply
Proposition~\pref{P:std-limits:properties} to $X_0^\mu\to S_0$ and
deduce that $X_\lambda^\mu\to S_\lambda$ has property $P$ for sufficiently
large $\lambda$ where $X_\lambda^\mu=X_0^\mu\times_{S_0}
S_\lambda$. As property~$P$ is stable under composition with closed
immersions, it follows that $\map{f_\lambda}{X_\lambda}{S_\lambda}$ has
property~$P$.
\end{proof}

\begin{lemma}\label{L:TheoremC-for-diagonal}
Let $S$ be a quasi-compact algebraic stack and let
$\{X_\lambda\to S\}_{\lambda\in L}$ be
an inverse system of quasi-compact and quasi-separated morphisms with affine
bonding maps $X_\mu\to X_\lambda$ and limit $X\to S$. Let $P$ be one of the
properties of \ref{I:prop-affine} or \ref{I:prop-closed}. If $\Delta_{X/S}$ has
property $P$, then $\Delta_{X_\lambda/S}$ has property $P$ for sufficiently
large $\lambda$. If the inertia $I_{X/S}$ is finite (resp.\ abelian,
  resp.\ tame), then so is the inertia $I_{X_\lambda/S}$ for sufficiently large
$\lambda$.
\end{lemma}
\begin{proof}
The proof is almost identical to the proof of
Corollary~\pref{C:properties-for-diagonal} replacing
Propositions~\pref{P:properties-for-closed-limits}
and~\pref{P:std-limits:properties} with the two lemmas above.
Recall that, by Lemma~\pref{L:diag-inv-system}, the diagonal
$\map{\Delta_{X/S}}{X}{X\times_S X}$ is the inverse limit of the morphisms
$X\times_{X_\lambda} X\to X\times_S X$ and that the bonding maps
$X\times_{X_\mu} X\to X\times_{X_\lambda} X$ are closed immersions. Note that
the morphisms $X\times_{X_\lambda} X\to X\times_S X$ are quasi-separated and of
finite type but not necessarily of finite presentation.
However, $X\times_{X_\lambda} X\to
X\times_S X$ is representable and thus of approximation type, fppf-locally on
the target. By Lemma~\pref{L:TheoremC-for-approx-type}, we deduce
that $X\times_{X_\lambda} X\to X\times_S X$ has property $P$ for sufficiently
large $\lambda$.

By Lemma~\pref{L:descent-for-finite-type-if-of-approx-type} we then deduce that
$X_\mu\times_{X_\lambda} X_\mu\to X_\mu\times_S X_\mu$ has property $P$ for all
sufficiently large $\mu$. As the bonding maps are affine, the diagonal
$\Delta_{X_\mu/X_\lambda}$ is a closed immersion. As property $P$ is stable
under composition with closed immersions, it follows that $\Delta_{X_\mu}$ has
property $P$ for all sufficiently large $\mu$.

The proof of the last statement about inertia is similar, cf.\ the proof
of Corollary~\pref{C:properties-for-diagonal}.
\end{proof}

\begin{proof}[Proof of Theorem~\tref{T:APPROXIMATION-PROPERTIES}]
We note that part~\ref{TI:C-scheme} is part~\ref{LI:TCat-scheme} of
Lemma~\pref{L:TheoremC-for-approx-type}. The other parts are fppf-local on $S$
so we can assume that $S$ is affine.

Each property of~\ref{I:prop-affine}, except for affine and quasi-affine,
corresponds to a property in~\ref{I:prop-affine} or~\ref{I:prop-closed} for the
diagonal or a property for the
inertia. Theorem~\tref{T:APPROXIMATION-PROPERTIES} for these properties
is thus an immediate consequence of Lemma~\pref{L:TheoremC-for-diagonal}.

Let $P$ be the property quasi-finite. We may then assume that $X_\alpha$ has
quasi-finite diagonal; hence there exists a quasi-finite flat presentation
$U_\alpha\to X_\alpha$~\cite[Thm.~7.1]{rydh_etale-devissage}. It is enough to
prove that $U_\lambda=U_\alpha\times_{X_\alpha} X_\lambda\to S$ is quasi-finite
for sufficiently large $\lambda$. We can thus replace $X_\lambda$ with
$U_\lambda$ and assume that $X_\lambda$ is an algebraic space.

Each of the remaining properties of~\ref{I:prop-affine} (affine, quasi-affine)
and each property of~\ref{I:prop-closed} except ``quasi-finite'' implies
that the diagonal is quasi-finite and locally separated.

Thus, for all the remaining properties, we can assume that $X_\lambda$ has
quasi-finite and locally separated diagonal. Then $X_\lambda\to S$ is of
approximation type (Corollary~\ref{C:qcqf=>global-type}). We may now conclude
the proof of
Theorem~\tref{T:APPROXIMATION-PROPERTIES} with
Lemma~\pref{L:TheoremC-for-approx-type}.
\end{proof}

%

\end{section}


\begin{section}{Applications}\label{S:applications}
The first application is a generalization of Chevalley's affineness theorem to
non-noetherian schemes and algebraic spaces. Also, we replace finite morphisms
by integral morphisms. Partial generalizations of this type for schemes have
been given by M.\ Raynaud~\cite[Prop.~3.1]{raynaud_sem-samuel} and B.\
Conrad~\cite[Cor.~A.2]{conrad_nagata}.

\begin{theorem}[Chevalley]\label{T:Chevalley-affineness}
Let $X$ be an affine scheme, let $Y$ be an algebraic space and let
$\map{f}{X}{Y}$ be an integral and surjective morphism. Then $Y$ is affine.
\end{theorem}
\begin{proof}
As $X$ is quasi-compact and $f$ is surjective it follows that $Y$ is
quasi-compact. As $f$ is universally closed and surjective and $X$ is
separated, it also follows that $Y$ is separated.

By Theorem~\tref{T:APPROXIMATION}, the morphism $\map{f}{X}{Y}$ has an
approximation $X\to X_0\to Y$ where $X_0\to Y$ is finite and finitely presented
and $X_0$ is affine. Replacing $X$ with $X_0$, we can thus assume that $f$ is
finitely presented.

By Theorem~\tref{T:APPROXIMATION}, we can write $Y$ as an inverse limit of
noetherian algebraic spaces $(Y_\lambda)_\lambda$ such that $Y\to Y_\lambda$ is
affine for every $\lambda$. Since $f$ is finitely presented, there is, by
standard limit methods, for sufficiently large $\lambda$, a finite surjective
morphism $\map{f_\lambda}{X_\lambda}{Y_\lambda}$ that pull-backs to
$\map{f}{X}{Y}$, cf.\ Appendix~\ref{A:std-limits}. After increasing $\lambda$
further, we can also assume that $X_\lambda$ is affine by
Theorem~\tref{T:APPROXIMATION-PROPERTIES}.
By Chevalley's theorem for finite morphisms between noetherian algebraic
spaces~\cite[Thm.~III.4.1]{knutson_alg_spaces}, it now follows that $Y_\lambda$
is affine and hence that $Y$ is affine.
\end{proof}

\begin{corollary}
Let $X$ be an algebraic space. If $X_\red$ is a scheme (resp.\ a quasi-affine
scheme, resp.\ an affine scheme),
then so is $X$.
\end{corollary}
\begin{proof}
If $X_\red$ is an affine scheme, then it follows by Chevalley's
theorem that $X$ is an affine scheme since $X_\red\inj X$ is finite
and surjective.
If $X_\red$ is a scheme, then there is an open covering $X=\bigcup U_i$
such that the $(U_i)_\red$ are affine and we conclude that the $U_i$ are
affine and that $X$ is a scheme.
If $X_\red$ is quasi-affine, then there exists an approximation
$X_\red\inj X_0\inj X$ where $X_0$ is quasi-affine and $X_0\inj X$ is
a finitely presented closed immersion (Theorem~\tref{T:APPROXIMATION}).
We have seen that $X$ is a scheme so~\cite[Prop.~4.5.13]{egaII} applies and
shows that $X$ is quasi-affine.
\end{proof}

\begin{theorem}\label{T:uni-closed=int+proper}
Let $X$ and $S$ be quasi-compact stacks with quasi-finite and separated
diagonals. Let $\map{f}{X}{S}$ be a universally closed, separated and
quasi-compact morphism.
Then $f$ factors through an integral surjective morphism $X\to X'$ followed by
a proper morphism $X'\to S$ with finite diagonal.
\end{theorem}
\begin{proof}
By Theorem~\tref{T:APPROXIMATION} there is an approximation $X\to X_0\to S$
where $X\to X_0$ is affine and $X_0\to S$ is of finite presentation with finite
diagonal. It follows that $X\to X_0$ is universally closed and thus
integral~\cite[Prop.~18.12.8]{egaIV}. Let $X'\inj X_0$ be the schematic image
of $X\to X_0$. Then $X\to X'$ is surjective and $X'\to S$ is universally
closed, hence proper.
\end{proof}

As an amusing corollary, we see that the finiteness assumption in Chevalley's
theorem on the fiber dimension can be removed.

\begin{corollary}[Chevalley]\label{C:Chevalleys-theorem-fiberdim}
Let $S$ be an algebraic stack and let $\map{f}{X}{S}$ be a universally closed,
separated and quasi-compact morphism with finite diagonal. Then, the fibers of
$f$ are finite-dimensional and the function $s\mapsto
\dim(f^{-1}(s))$ is upper semi-continuous.
\end{corollary}
\begin{proof}
Using Theorem~\pref{T:uni-closed=int+proper} and
Theorem~\tref{T:FINITE-COVERINGS} we can assume that $f$ is
proper and strongly representable. This case follows
from~\cite[Cor.~13.1.5]{egaIV}.
\end{proof}

The following theorem settles a question of
Grothendieck~\cite[Rem.\ 18.12.9]{egaIV}. This was also the original motivation
for this paper.

\begin{theorem}\label{T:char-of-integral-morphisms}
Let $\map{f}{X}{S}$ be a morphism of algebraic stacks. Then $f$ is integral
if and only if $f$ is universally closed, separated and has affine fibers.
\end{theorem}
\begin{proof}
Taking a presentation, we can assume that $S$ is affine. The
necessity follows from~\cite[Cor.~6.1.10]{egaII} so assume that $f$ is
universally closed, separated and has affine fibers.
Note that $f$ is representable and quasi-compact, since the fibers of $f$ are
quasi-compact and $f$ is closed. Thus, by
Theorem~\pref{T:uni-closed=int+proper}, there is a factorization of $f$ into an
integral surjective morphism $X\to X'$ followed by a proper
and representable morphism
$X'\to S$. Chevalley's theorem~\pref{T:Chevalley-affineness} then shows that
the fibers of $X'\to S$ are affine and hence finite. As a quasi-finite and
proper morphism is finite~\cite[Cor.~A.2.1]{laumon}, the theorem follows.
\end{proof}


The following variant of Zariski's Main Theorem
generalizes~\cite[Cor.~18.12.13]{egaIV},~\cite[Thm.\
II.6.15]{knutson_alg_spaces} and~\cite[Thm.~16.5]{laumon}.


\begin{theorem}[Zariski's Main Theorem]\label{T:ZMT}
Let $S$ be an algebraic stack and let $\map{f}{X}{S}$ be a representable,
quasi-finite and separated morphism. Then
\begin{enumerate}
\item There is a factorization $X\to X'\to S$ of $f$ where $X\to X'$
is an open immersion and $X'\to S$ is integral.\label{TI:ZMT-1}
\item If $S$ is pseudo-noetherian (or at least has the completeness property),
then there exists a factorization as above with $X'\to S$ finite. If, in
addition, $f$ is of finite presentation, we can choose $X'\to S$ to be of finite
presentation.\label{TI:ZMT-2}
\end{enumerate}
\end{theorem}
\noindent Note that \ref{TI:ZMT-2} is satisfied if $S$ is noetherian or of
approximation type, e.g., quasi-compact with quasi-finite and locally separated
diagonal.
\begin{proof}[Proof of Theorem~\pref{T:ZMT}]
\ref{TI:ZMT-1} (cf.\ \cite[Thm.~16.5]{laumon}) The integral closure of
$\sO_S\to f_*\sO_X$ is a quasi-coherent $\sO_S$-subalgebra $\sA\subseteq
f_*\sO_X$. Let $X'$ be the spectrum of $\sA$. This gives a factorization $X\to
X'\to S$ where the first morphism is quasi-finite, representable, separated,
schematically dominant and integrally closed and the second morphism is
integral. It follows that $X\to X'$ is an open immersion by
Lemma~\pref{L:qfin+int-closed=open}.

\ref{TI:ZMT-2} Write $X'=\varprojlim_\lambda X'_\lambda$ as a limit of finite
and finitely presented morphisms $X'_\lambda\to S$. For sufficiently large
$\lambda$, there exists, by Remark~\pref{R:std-limits:qc-open}, an open substack
$X_\lambda\subseteq X'_\lambda$ such that $X=X_\lambda\times_{X'_\lambda}
X'$. As $f$ is of finite type so is $X\to X_\lambda$ and thus we have
that $X\to X_\lambda$ is finite. If $f$ is of finite presentation, then so is
$X\to X_\lambda$.  Since $X'_\lambda$ has the completeness property, it has the
extension property~\ref{I:E2} for the category of integral algebras. This gives
the existence of a cartesian diagram
$$\xymatrix{X\ar@{=}[r]\ar[d] & X \ar[r]\ar[d] & X_\lambda\ar[d]\\
X'\ar[r] & X''\ar[r]\ar@{}[ul]|\square & X'_\lambda\ar@{}[ul]|\square}$$
where $X''\to X'_\lambda$ is finite (resp.\ finite and of finite
presentation). The requested factorization is $X\to X''\to S$.
\end{proof}

\begin{theorem}[Serre's criterion]\label{T:serre}
Let $S$ be a quasi-compact algebraic stack with quasi-finite and separated
diagonal and let
$\map{f}{X}{S}$ be a representable quasi-compact and quasi-separated
morphism.
Then $f$ is affine if and only if $\map{f_*}{\QCoh(X)}{\QCoh(S)}$
is exact (and if and only if $f_*$ is faithfully exact).
\end{theorem}
\begin{proof}
If $f$ is affine, then $f_*$ is faithfully exact. Indeed, this can be checked
on a presentation $\map{g}{S'}{S}$ since $g^*$ is faithfully exact.
Conversely, assume that $f_*$ is exact. Let $\map{g}{S'}{S}$ be a finite
surjective morphism with $S'$ a scheme as in Theorem~\tref{T:FINITE-COVERINGS}
and let $\map{f'}{X'}{S'}$ be the pull-back of $f$. Then since $g_*$ is
faithfully exact it follows that $f'_*$ is exact.


Let $\map{h}{X''}{X'}$ be a finite surjective morphism with $X'$ a scheme, cf.\
Theorem~\tref{T:FINITE-COVERINGS}. Then $f'_*h_*$ is exact and it
follows from Serre's criterion for schemes that $\map{f'\circ h}{X''}{S'}$ is
affine~\cite[Cor.~5.2.2]{egaII},~\cite[1.7.18]{egaIV}. In particular, we have that $g\circ f'\circ h$ is affine. It then
follows from Chevalley's theorem~\pref{T:Chevalley-affineness} that $f$ is
affine.
\end{proof}

Using~\cite[Prop.~3.10 (vii)]{alper_good-mod-spaces}, one may strengthen
Theorem~\pref{T:serre}, only requiring that $S$ is a, not necessarily
quasi-compact, algebraic stack with quasi-affine diagonal.

The following forms of Chow's lemma generalize~\cite[Cor.~16.6.1]{laumon}
and~\cite[Cor.~5.7.13]{raynaud-gruson}. Note that the hypothesis that $\stX/S$
is separated is missing in the statement of~\cite[Cor.~16.6.1]{laumon}.

\begin{theorem}[Chow's lemma]
Let $S$ be a quasi-compact and quasi-separated algebraic space, let $X$ be a
quasi-compact stack with quasi-finite and separated diagonal and let
$\map{f}{X}{S}$ be a morphism of finite presentation. Then there exists a
commutative diagram
$$\xymatrix{%
X\ar[d]^f & X'\ar[l]^p\ar[d]^g\\
S & P\ar[l]^\pi}$$
of finitely presented morphisms such that:
\begin{enumerate}
\item $X'$ is a scheme;
\item $\pi$ is projective;
\item $p$ is proper, strongly representable and surjective; and
\item $g$ is strongly representable and \etale{} (but not necessarily separated).
\end{enumerate}
If $f$ is separated, then $g$ can be chosen to be
an open immersion (so $\pi\circ g$ is quasi-projective).
\end{theorem}
\begin{proof}
By Theorem~\tref{T:APPROXIMATION}, we can assume that $S$ is noetherian.
Replacing
$X$ with a finite cover as in Theorem~\tref{T:FINITE-COVERINGS}, we can assume
that $X$ is a scheme. The result then follows
from~\cite[Cor.~5.7.13]{raynaud-gruson}.
\end{proof}

We have the following variant of the previous result which is more in the
spirit of the usual Chow's lemma for schemes. In this statement we can also
drop the finite presentation hypothesis.

\begin{theorem}[Chow's lemma]
Let $S$ be a quasi-compact and quasi-separated algebraic space, let $X$ be a
quasi-compact stack with quasi-finite and separated diagonal and a \emph{finite
number of irreducible components}. Let $\map{f}{X}{S}$ be a morphism of finite
type. Then there exists a commutative diagram
$$\xymatrix{%
X\ar[d]^f & X'\ar[l]^p\ar[d]^g\\
S & P\ar[l]^\pi}$$
such that:
\begin{enumerate}
\item $X'$ is a scheme;
\item $\pi$ is projective;
\item $p$ is proper, strongly representable and surjective and there exists
a dense open substack $U\subseteq X$ such that $p|_U$ is finite, flat and
finitely presented;
\item $g$ is \etale{} and strongly representable.
\end{enumerate}
Moreover, if $f$ is separated, then $g$ can be chosen as an open
immersion and if $X$ is Deligne--Mumford, then $p|_U$ can be taken to be
\etale{}.
\end{theorem}
\begin{proof}
Replacing $X$ with a finite generically \etale{} (resp.\ generically flat)
cover as in Theorem~\tref{T:FINITE-COVERINGS}, we can assume that $X$ is
a scheme. The result then follows
from~\cite[Cor.~5.7.13]{raynaud-gruson}.
\end{proof}



The following theorem partly generalizes
Proposition~\pref{P:std-limits:properties} from finitely presented morphisms to
quasi-separated morphisms of finite type.

\begin{theorem}\label{T:std-limits:properties:finite-type}
Let $S_0$ be a quasi-compact algebraic stack and let $S=\varprojlim_\lambda
S_\lambda$ be the limit of an inverse system of stacks that are affine over
$S_0$. Let
$\alpha$ be an index and let $\map{f_\alpha}{X_\alpha}{S_\alpha}$ be a
quasi-separated morphism of finite type. For every $\lambda>\alpha$,
let $\map{f_\lambda}{X_\lambda}{S_\lambda}$ (resp.\ $\map{f}{X}{S}$) be the
base change of $f_\alpha$ along $S_\lambda\to S_\alpha$ (resp.\ $S\to
S_\alpha$). Let $P$ be one of the properties of~\ref{I:prop-affine}
or~\ref{I:prop-closed}.
Then $f$ (resp.\ $\Delta_f$) has property $P$ if and only if $f_\lambda$
(resp.\ $\Delta_{f_\lambda}$) has property $P$ for all sufficiently large
$\lambda$'s.
\end{theorem}
\begin{proof}
The question is fppf-local on $S_0$ so we can assume that $S_0$ is affine. Note
that,
by assumption, the diagonal $\Delta_{f_\lambda}$ is of finite presentation for
every $\lambda$. If the property is a property of the diagonal (e.g., the property ``separated'' corresponds to the property ``proper'' for the diagonal)
we deduce the theorem from Proposition~\pref{P:std-limits:properties} applied to
$\Delta_{f_\lambda}$.

If $P$ is quasi-finite, then we may find a quasi-finite flat presentation
$\map{p}{U}{X}$~\cite[Thm.~7.1]{rydh_etale-devissage} with $U$ affine. By
standard limit methods there is, for sufficiently large $\lambda$, an affine
scheme $U_\lambda$ and a quasi-finite flat presentation
$\map{p_\lambda}{U_\lambda}{X_\lambda}$ that is pulled back to $p$. It is
enough to show that $U_\lambda\to S_\lambda$ is quasi-finite for sufficiently
large $\lambda$, so we can replace $\alpha$ with $\lambda$, $X_\alpha$ with
$U_\lambda$ and assume that $X_\alpha$ is affine.

If $P$ is any of the other remaining properties, then the diagonal is
quasi-finite and locally separated. We can thus apply
Proposition~\pref{P:std-limits:properties} to deduce that, for sufficiently
large $\lambda$, the diagonal $\Delta_{f_\lambda}$ is quasi-finite and locally
separated. After increasing $\alpha$ we can thus assume that
$X_\alpha$ is of global type and hence can be approximated.

The theorem now follows from
Lemma~\pref{L:descent-for-finite-type-if-of-approx-type}.
\end{proof}

Theorem~\pref{T:std-limits:properties:finite-type} does not hold for properties
that are not stable under closed
immersions such as: isomorphism, open immersion, \etale{}, finite presentation,
surjective, flat, smooth, universally subtrusive, universally open.
An easy counter-example
is furnished by taking $S_0=\Spec(k[x_1,x_2,\dots])$ and $S=X_0$ as the
inclusion of the origin.

Recall that a morphism $\map{f}{X}{Y}$ is of \emph{constructible finite type}
if $f$ is of finite type and quasi-separated, and for any morphism $Y'\to Y$,
the base change $\map{f'}{X\times_Y Y'}{Y'}$ maps ind-constructible subsets
onto ind-constructible subsets~\cite[App.~D]{rydh_embeddings-of-unramified}.
The following result was surmised in~\loccit.

\begin{proposition}
Let $Y$ be a pseudo-noetherian stack and let $\map{f}{X}{Y}$ be a morphism of
approximation type (e.g., let $X$ and $Y$ be stacks of global type). Then $f$
is of constructible finite type if and only if $f$ can be factored as a
nil-immersion $X\inj X'$ followed by a finitely presented morphism $X'\to Y$.
\end{proposition}
\begin{proof}
As nil-immersions and finitely presented morphisms are of construct\-ible
finite type the condition is sufficient. To see that it is necessary write
$X=\varprojlim_\lambda X_\lambda$ as an inverse limit of finitely presented
morphisms $X_\lambda\to Y$ with closed immersions as bonding maps. By
assumption $X\inj X_\lambda$ is of constructible finite type so
$|X|\subseteq |X_\lambda|$ is constructible. It follows that $X\inj X_\lambda$
is bijective for sufficiently large $\lambda$ by~\cite[Cor.~8.3.5]{egaIV}.
\end{proof}

\end{section}


\appendix
\begin{section}{Constructible properties}\label{A:constructible}
In this appendix, we extend standard results on constructible properties for
schemes to algebraic stacks. We also show that if $G\to S$ is a group scheme,
then the locus of points $s\in S$ with an abelian (resp.\ a finite and linearly
reductive) fiber $G_s$, is open.

\begin{lemma}\label{L:fiberwise-conds}
Let $Y$ be an algebraic space. Let $\map{f}{X}{Y}$ be a morphism between
algebraic stacks that is locally of finite type. For a point $y\in |Y|$,
let $\map{f_y}{X_y}{\Spec \kappa(y)}$ denote the fiber. Then
\begin{enumerate}
\item $f$ is representable if and only if the fiber $f_y$ is representable
for every $y\in |Y|$;
\item $f$ is a monomorphism if and only if the fiber $f_y$ is a monomorphism
for
every $y\in |Y|$ (i.e., $X_y$ is either empty or $\kappa(y)$-isomorphic to
$\Spec \kappa(y)$ for every $y\in |Y|$); and
\item $f$ is unramified if and only if the fiber $f_y$ is unramified for
every $y\in |Y|$.
\end{enumerate}
\end{lemma}
\begin{proof}
(iii) is~\cite[Prop.~B.3]{rydh_embeddings-of-unramified}.
The necessity of (i) and (ii) is clear. We begin with showing the sufficiency
of (ii) under the additional assumption that $f$ is representable. If $f_y$ is
a monomorphism for every $y\in |Y|$, then $f$ is unramified by (iii) and, in
particular, $\Delta_f$ is an open immersion. As $f$ is universally injective, we
have that $\Delta_f$ is bijective and hence an isomorphism, i.e., $f$ is a
monomorphism.

Now assume that $f_y$ is representable for every $y\in |Y|$. Then
$(\Delta_{X/Y})_y$ is a monomorphism for every $y\in |Y|$ and hence
$\Delta_{X/Y}$ is a monomorphism by the special case of (ii). This shows that
$f$ is representable. The general case of (ii) now follows from (i) and the
special case of (ii).
\end{proof}

The following lemma generalizes~\cite[Prop.~9.2.6.1]{egaIV} to
non-representable morphisms.

\begin{lemma}\label{L:dim-is-constructible}
Let $S$ be an algebraic space and let $X$ be an algebraic stack of finite
presentation over $S$. The function $|S|\to \Z$ defined by $s\mapsto \dim X_s$
is constructible.
\end{lemma}
\begin{proof}
Let $\map{f}{X}{S}$ denote the structure morphism. Let $\map{x}{\Spec k}{X}$
be a point and let $s=f\circ x$. The dimension $\dim_x X_s$ only depends on the
image of the point $x$ in the topological space $|X|$. This gives a function
$\map{\dim_{X/S}}{|X|}{\Z}$. Let $\map{p}{U}{X}$ be a smooth presentation. Then
$\map{\dim_{U/S}}{|U|}{\N}$ is constructible~\cite[Prop.~9.9.1]{egaIV} (and
upper semi-continuous) and $\map{\dim_{U/X}}{|U|}{\N}$ is locally constant.
Thus $\dim_{X/S}\circ |p|=\dim_{U/S}-\dim_{U/X}$ is a constructible (and upper
semi-continuous) function. The set $f\bigl(\dim_{X/S}^{-1}(d)\bigr)=
(f\circ p)\bigl((\dim_{X/S}\circ |p|)^{-1}(d)\bigr)$ is thus constructible by Chevalley's
theorem~\cite[Thm.~7.1.4]{egaI_NE}. Since
$$\dim X_s=\sup_{x\in |X_s|} \dim_x
X_s=\sup_d \Bigl\{d\;:\; s\in f\bigl(\dim_{X/S}^{-1}(d)\bigr)\Bigr\}$$
it follows that
$s\mapsto \dim X_s$
is constructible.
\end{proof}

\begin{proposition}\label{P:constructible-conds}
Let $S$ be an algebraic space and let $\map{f}{X}{Y}$ be a morphism between
algebraic stacks of finite presentation over $S$. Let $P$ be one of the
following properties of a morphism:
\begin{enumerate}
\item monomorphism,\label{TI:cc:mono}
\item universally injective (i.e., ``radiciel''),\label{TI:cc:ui}
\item surjective,\label{TI:cc:surj}
\item isomorphism,\label{TI:cc:iso}
\item representable,\label{TI:cc:repr}
\item unramified,\label{TI:cc:unram}
\item flat,\label{TI:cc:flat}
\item \etale{},\label{TI:cc:etale}
\item quasi-finite,\label{TI:cc:qf}
\item has quasi-finite diagonal.\label{TI:cc:qf-diag}
\end{enumerate}
Then, the set of points $s\in |S|$ such that $\map{f_s}{X_s}{Y_s}$ has $P$ is
constructible.
\end{proposition}
\begin{proof}
The question is fppf-local on $S$ so we can assume that $S$ is affine. We may
also replace $Y$ with a presentation and assume that $Y$ is affine.
When $f$ is strongly representable, the proposition holds
by~\cite[Props.~9.6.1,~11.2.8 and~17.7.11]{egaIV}.

If $f$ is representable, then let $X'\to X$ be an \etale{} presentation with
$X'$ a scheme. The corresponding result for $X'\to X\to Y$ implies the result
for
$f$ and all properties with the exception of
\ref{TI:cc:mono} monomorphism, \ref{TI:cc:ui} universally injective
and~\ref{TI:cc:iso} isomorphism.
The locus where $f$ is a monomorphism (resp.\ universally injective) coincides
with the locus where $\Delta_f$ is an isomorphism (resp.\ surjective) and this
locus
is constructible (since $\Delta_f$ is strongly representable). This settles
properties~\ref{TI:cc:mono} and~\ref{TI:cc:ui}. Finally, $f$ is an isomorphism
if and only if $f$ is a surjective \etale{} monomorphism, so
property~\ref{TI:cc:iso} is constructible.

For general $f$, we can now deduce that the proposition holds for the
properties: \ref{TI:cc:mono} monomorphism, \ref{TI:cc:ui} universally
injective, \ref{TI:cc:repr} representable, \ref{TI:cc:unram} unramified and
\ref{TI:cc:qf-diag} quasi-finite diagonal; by considering the corresponding
properties for the diagonal: \ref{TI:cc:iso} isomorphism, \ref{TI:cc:surj}
surjective, \ref{TI:cc:mono} monomorphism, \ref{TI:cc:etale} \etale{}, and
\ref{TI:cc:qf} quasi-finite.
Properties~\ref{TI:cc:surj} surjective and~\ref{TI:cc:flat} flat, follow by
taking a presentation $X'\to X$. Property~\ref{TI:cc:etale} \etale{}, is
the conjunction of properties~\ref{TI:cc:unram} unramified and~\ref{TI:cc:flat}
flat. As before, property~\ref{TI:cc:iso} isomorphism, is the conjunction of
properties~\ref{TI:cc:mono}, \ref{TI:cc:surj} and~\ref{TI:cc:etale}.

Property~\ref{TI:cc:qf} quasi-finite can be checked on fibers and we can 
thus, by Chevalley's Theorem~\cite[Thm.~7.1.4]{egaI_NE}, assume that $S=Y$.
The set in question is then the set where the fibers of $f$ and $\Delta_f$
both have dimension zero. This set is constructible by
Lemma~\pref{L:dim-is-constructible}.
\end{proof}

\begin{proposition}\label{P:constructible-conds-2}
Let $S$ be an algebraic space and let $G$ be an $S$-group space of finite
presentation (i.e., a group object in the category of algebraic spaces). The
set of points $s\in |S|$ such that $G_s$ is abelian (resp.\ finite and linearly
reductive) is constructible.
\end{proposition}
\begin{proof}
The question is fppf-local on $S$ so we can assume that $S$ is a scheme. Let
$G$ act on itself
by conjugation and let $\map{\rho}{G\times_S G}{G}$ be the corresponding
morphism, pointwise given by $(g,h)\mapsto ghg^{-1}$. As the diagonal
$\map{\Delta}{G}{G\times_S G}$ is of finite presentation, the subset
$Z\subseteq |G\times_S G|$ where $\rho=\pi_2$ is constructible.
As the structure morphism
$\map{p}{G\times_S G}{S}$ is of finite presentation, it follows that the subset
$W=S\setminus p(G\times_S G\setminus Z)$, of points $s\in |S|$ such that $G_s$ is
abelian, is constructible.

For a group scheme $H\to \Spec k$, we let $E(H,k)$ be the property that
$H\to\Spec k$ is finite and linearly reductive, or, equivalently, that
$H\to\Spec k$ is \emph{locally
well-split}~\cite[Prop.~2.10]{abramovich_olsson_vistoli_tame_stacks}. This
property is stable under field extensions $k'/k$. Let $E$ be the set of points
$s\in |S|$ such that $E(G_s,\kappa(s))$ holds. We have to show that $E$ is
constructible. By~\cite[Prop.~9.2.3]{egaIV}, it is enough to show that if $S$
is an integral noetherian scheme, then there is an open dense subset
$U\subseteq S$ that is contained in either $E$ or~$S\setminus E$.

To show this we can replace $S$ with an open dense subset such that $G\to S$
becomes flat~\cite[Thm.~11.1.1]{egaIV}. Moreover, as the property of having
finite fibers is constructible, we can assume that $G\to S$ is
quasi-finite. After replacing $S$ with an open dense subset, we can further
assume that
$G\to S$ is finite. Then $E$ is open
by~\cite[Lem.~2.13]{abramovich_olsson_vistoli_tame_stacks} and thus either $E$
is dense or empty.
\end{proof}

\end{section}

\begin{section}{Standard limit results}\label{A:std-limits}
In this appendix, we generalize the standard limit methods for schemes
in~\cite[\S8]{egaIV} to algebraic stacks.
This has been done in~\cite[Props.~4.15, 4.18]{laumon} (also
see~\cite[Prop.~2.2]{olsson_Hom-stacks}) for algebraic stacks over inverse
systems of \emph{affine} schemes. In this appendix, we allow inverse system of
algebraic stacks.
As elsewhere, we do not insist that the diagonal of an algebraic stack is
separated. All inverse systems are assumed to be filtered and to have affine
bonding maps so their inverse limits exist in the category of algebraic
stacks.

We begin with the functorial characterization of morphisms that
are locally of finite presentation, cf.\ \cite[Prop.~8.14.2]{egaIV}.

\begin{proposition}\label{P:functorial-char-of-loc-fin-pres}
Let $\map{f}{Y}{S}$ be a morphism of algebraic stacks. The following are
equivalent.
\begin{enumerate}
\item $f$ is locally of finite presentation.
\item For every inverse system $\{\map{g_\lambda}{X_\lambda}{S}\}$ of
quasi-compact and quasi-separated stacks $X_\lambda$ with limit
$\map{g}{X}{S}$ the functor
$$\varinjlim_{\lambda} \catHom_S(X_\lambda,Y)\to \catHom_S(X,Y)$$
is an equivalence of categories.\label{PI:fclfp-2}
\item As~\ref{PI:fclfp-2} but with $X_\lambda$ affine for every $\lambda$.
\end{enumerate}
\end{proposition}
\begin{proof}
Clearly (ii)$\implies$(iii). That (i)$\iff$(iii) is~\cite[Prop.~4.15
(i)]{laumon}. Let us show that (iii)$\implies$(ii) which essentially is the
proof of~\cite[Prop.~4.18 (i)]{laumon}. After making the base change
$X_\lambda\to S$ for some $\lambda$, we can assume that the $X_\lambda$'s are
affine over $S$. Let $U_0\to S$ be a presentation and let
$U_\lambda=U_0\times_S X_\lambda$; then $U_\lambda\to X_\lambda$ is a
presentation. Let
$(U_\lambda/X_\lambda)^i=U_\lambda \times_{X_\lambda}\dots\times_{X_\lambda}
U_\lambda$ denote the $i$\textsuperscript{th} fiber product.
The category $\catHom_S(X_\lambda,Y)$ is equivalent to the
category given by the cosimplicial diagram of categories
$$\xymatrix{\catHom_S(U_\lambda,Y)\ar@<.7ex>[r]\ar@<-.7ex>[r] &
\catHom_S\bigl((U_\lambda/X_\lambda)^2,Y\bigl)\ar@{+->+}[l]\ar@<1.4ex>[r]\ar[r]\ar@<-1.4ex>[r] &
\catHom_S\bigr((U_\lambda/X_\lambda)^3,Y\bigr)\ar@<.7ex>@{+->+}[l]\ar@<-.7ex>@{+->+}[l]}$$
(cf.\ \loccit)\ and this construction commutes with filtered colimits. It is
therefore enough to show the proposition after replacing $X_\lambda$ with
$(U_\lambda/X_\lambda)^i$ for $i=1,2,3$.

Firstly, assume that $S$ is a separated algebraic space, and choose a
presentation
$U_0\to S$ with $U_0$ affine. Then the fiber products $(U_\lambda/X_\lambda)^i$
are affine for $i=1,2,3$ and we are done in this case. Secondly, assume that
$S$ is
an algebraic space. Then $(U_\lambda/X_\lambda)^i$ are separated algebraic
spaces and this case follows from the previous. Thirdly, assume that $S$ is
a general algebraic stack. Then $(U_\lambda/X_\lambda)^i$ are algebraic spaces
and this settles the final case.
\end{proof}

%
\begin{proposition}\label{P:std-limits}
Let $S_0$ be an algebraic stack and let $S=\varprojlim_{\lambda} S_\lambda$ be
an inverse limit of stacks that are affine over $S_0$.
\begin{enumerate}
\item Let $X_0\to S_0$ and $Y_0\to S_0$ be morphisms of stacks and let
$$X_\lambda=X_0\times_{S_0} S_\lambda,\quad
Y_\lambda=Y_0\times_{S_0} S_\lambda,$$
$$X=X_0\times_{S_0} S,\quad Y=Y_0\times_{S_0} S$$
for every $\lambda$. Suppose that $X_0$ is quasi-compact and quasi-separated
and that $Y_0\to S_0$ is locally of finite presentation. Then, the functor
$$\varinjlim_{\lambda} \catHom_{S_\lambda}(X_\lambda,Y_\lambda)
\to \catHom_S(X,Y)$$
is an equivalence of categories.
\item Suppose that $S_0$ is quasi-compact and quasi-separated. Let $X\to S$ be
a morphism of finite presentation. Then, there exists an index $\alpha$, an
algebraic stack $X_\alpha$ of finite presentation over $S_\alpha$ and an
$S$-isomorphism $X_\alpha\times_{S_\alpha} S\to X$.
\end{enumerate}
\end{proposition}
\begin{proof}
Note that $\catHom_{S_\lambda}(X_\lambda,Y_\lambda)=
\catHom_{S_0}(X_\lambda,Y_0)$
and $\catHom_{S}(X,Y)=\catHom_{S_0}(X,Y_0)$. The first statement thus
follows from Proposition~\pref{P:functorial-char-of-loc-fin-pres} with
$S=S_0$ and $Y=Y_0$.

(ii) When $S_0$ and $X$ are schemes, this is~\cite[Thm.~8.8.2 (ii)]{egaIV}. The
extension to the case where $X$ is an algebraic space is not difficult, cf.\
\cite[Prop.~4.18]{laumon}. For the general case, choose a presentation
$V_0\to S_0$ with $V_0$ affine and let $V_\lambda=V_0\times_{S_0} S_\lambda$
and $V=V_0\times_{S_0} S$. Also choose a presentation $U\to X\times_S V$ and
let $R=U\times_X U$. Then $X$ is the quotient of the groupoid
$[\equalizer{R}{U}]$. Consider $R$ as a $V\times_S V$-space.

Applying the case with algebraic spaces, there is an index $\lambda$ and
algebraic spaces $U_{\lambda}$ and $R_{\lambda}$ of finite presentation over
$V_\lambda$ and $V_\lambda\times_{S_\lambda} V_\lambda$ such that their
pull-backs to $V$ and $V\times_S V$ are isomorphic to $U\to V$ and $R\to
V\times_S V$. By composition, we obtain finitely presented morphisms $U_\lambda
\to S_\lambda$ and $R_\lambda\to S_\lambda$ such that their pull-backs are
isomorphic to
$U\to S$ and $R\to S$. Note that the last statement takes place in the
$1$-category
$\AlgSp_{/S}$ and that isomorphic signifies that there are $2$-commutative
diagrams
$$\vcenter{\xymatrix{{U_\lambda\times_{S_\lambda} S}\rrlowertwocell_{\pi_2}{}\ar[r]^-{\iso} & U\ar[r]& S}} \quad 
\vcenter{\xymatrix{{R_\lambda\times_{S_\lambda} S}\rrlowertwocell_{\pi_2}{}\ar[r]^-{\iso} & R\ar[r]& S.}}$$
By (i), we can, for sufficiently large $\lambda$, find morphisms such that we
obtain a groupoid $\equalizer{R_\lambda}{U_\lambda}$ in the category
$\AlgSp_{/S_\lambda}$ which pull-backs to the groupoid $\equalizer{R}{U}$ in
the category $\AlgSp_{/S}$. For sufficiently large $\lambda$, we can also assume
that the morphisms $\map{s,t}{R_\lambda}{U_\lambda}$ are smooth. Indeed, this
can be checked on an \etale{} presentation of the algebraic space $R_\lambda$
so we can apply~\cite[Prop.~17.7.8]{egaIV}.

Let $X_\lambda=[\equalizer{R_\lambda}{U_\lambda}]$ be the quotient stack. Then,
there
is an induced finitely presented morphism $X_\lambda\to S_\lambda$, unique up to
unique
$2$-isomorphism, such that the pull-back $X_\lambda\times_{S_\lambda} S\to S$ is
isomorphic to $X\to S$.
%
\end{proof}

\begin{proposition}\label{P:std-limits:properties}
Let $S_0$ be a quasi-compact algebraic stack and let $S=\varprojlim_\lambda
S_\lambda$ be an inverse limit of stacks that are affine over $S_0$. Let
$\alpha$ be an index and let $\map{f_\alpha}{X_\alpha}{Y_\alpha}$ be a morphism
between
stacks of finite presentation over $S_\alpha$. For every $\lambda>\alpha$,
let $\map{f_\lambda}{X_\lambda}{Y_\lambda}$ (resp.\ $\map{f}{X}{Y}$) be the
base change of $f_\alpha$ along $S_\lambda\to S_\alpha$ (resp.\ $S\to
S_\alpha$).  Let $P$ be one of the following properties of a morphism:
\begin{enumerate}
\item representable,\label{PI:limit:repr}\label{PI:limit:first-1}
\item a monomorphism,\label{PI:limit:mono}
\item an isomorphism,\label{PI:limit:iso}\label{PI:limit:first-mono}
\item an immersion,\label{PI:limit:imm}
\item a closed immersion,\label{PI:limit:closed-imm}
\item an open immersion,\label{PI:limit:last-mono}
\item universally injective (i.e., ``radiciel''),\label{PI:limit:ui}
\item a universal homeomorphism,\label{PI:limit:uh}
\label{PI:limit:last-1}
\item surjective,\label{PI:limit:surjective}
\item flat,\label{PI:limit:flat}
\item universally subtrusive,\label{PI:limit:univ-subt}
\item universally open,\label{PI:limit:univ-open}
\item smooth,\label{PI:limit:smooth}
\item unramified,\label{PI:limit:unramified}\label{PI:limit:first-2}
\item \etale{},\label{PI:limit:etale}
\item locally separated,\label{PI:limit:loc-sep}\label{PI:limit:last-2}
\item separated,\label{PI:limit:sep}
\item proper,\label{PI:limit:proper}
\item affine,\label{PI:limit:aff}
\item quasi-affine,\label{PI:limit:q-aff}
\item finite, or\label{PI:limit:finite}
\item quasi-finite.\label{PI:limit:first-fiber}
\end{enumerate}
Then $f$ (resp.\ $\Delta_f$) has property $P$ if and only if $f_\lambda$
(resp.\ $\Delta_{f_\lambda}$) has property $P$ for all sufficiently large
$\lambda$'s.

If, in addition, $X_\alpha\to Y_\alpha$ is representable and a group object,
then
the same conclusion holds for the properties:
\begin{enumerate}\setcounter{enumi}{22}
\item abelian fibers, and
\item quasi-finite with linearly reductive fibers.\label{PI:limit:last-fiber}
\end{enumerate}
\end{proposition}
\begin{proof}
The condition is clearly sufficient as all properties are stable under base
change. We will prove that it is necessary.
We can assume that $S_0=S_\alpha=Y_\alpha$. As the properties are fppf-local on
the base, we can further assume that $S_0$ is an affine scheme. When $f_\alpha$
is strongly representable, the proposition is~\cite[Thms.~8.10.5, 11.2.6,
Prop.~17.7.8]{egaIV} and~\cite[Thms.~6.4 and~6.6]{rydh_submersion_and_descent}
(for properties~\ref{PI:limit:univ-subt} and~\ref{PI:limit:univ-open}).

\ref{PI:limit:surjective}--\ref{PI:limit:smooth}:
Properties~\ref{PI:limit:surjective} surjective, \ref{PI:limit:flat} flat,
\ref{PI:limit:univ-subt} universally subtrusive, \ref{PI:limit:univ-open}
universally open and \ref{PI:limit:smooth} smooth can be checked after
replacing $X_\alpha$ with a smooth presentation.

\ref{PI:limit:first-1}--\ref{PI:limit:last-1},
\ref{PI:limit:first-2}--\ref{PI:limit:last-2}:
Assume that the proposition has been proven when $f_\alpha$ is representable
(resp.\ strongly representable). Then, for general $f_\alpha$ (resp.\
representable $f_\alpha$), we note that properties: \ref{PI:limit:repr}
representable, \ref{PI:limit:mono} monomorphism, \ref{PI:limit:ui} universally
injective, \ref{PI:limit:unramified} unramified and~\ref{PI:limit:loc-sep}
locally separated; correspond respectively to the properties:
\ref{PI:limit:mono} monomorphism, \ref{PI:limit:iso} isomorphism,
\ref{PI:limit:surjective} surjective, \ref{PI:limit:etale} \etale{}
and~\ref{PI:limit:imm} immersion; of the diagonal which is representable
(resp.\ strongly representable).
A monomorphism is strongly representable
by~\cite[Thm.~6.15]{knutson_alg_spaces} and hence
properties~\ref{PI:limit:first-mono}--\ref{PI:limit:last-mono} follows
from~\ref{PI:limit:mono} and the strongly representable
case. Property~\ref{PI:limit:etale} is the conjunction of~\ref{PI:limit:flat}
and~\ref{PI:limit:unramified}.
Likewise, property~\ref{PI:limit:uh} is the conjunction of
properties~\ref{PI:limit:ui},~\ref{PI:limit:surjective}
and~\ref{PI:limit:univ-open}.

\ref{PI:limit:sep} and~\ref{PI:limit:proper}:
Assume that the proposition has been proven when $f_\alpha$ is representable
(resp.\ strongly representable). Then, for general $f_\alpha$ (resp.\
representable $f_\alpha$) the proposition holds for property
\ref{PI:limit:sep} separated, by considering the diagonal.
Let $f$ be proper. We can then assume that $f_\alpha$ is separated. Writing
$S_\alpha$ as a limit $\varprojlim_{\beta} S_{\alpha\beta}$ of noetherian
affine schemes, we can, by Proposition~\pref{P:std-limits}, assume that
$f_\alpha$ is the pull-back of a finitely presented separated morphism
$X_{\alpha\beta}\to S_{\alpha\beta}$. Then, by Chow's lemma~\cite[Ch.~IV,
Thm.~3.1]{knutson_alg_spaces} and~\cite{olsson_proper-coverings}, there exists
a \emph{scheme} $Z_{\alpha\beta}$ and a proper surjective morphism
$Z_{\alpha\beta}\to X_{\alpha\beta}$. Let
$Z_{\alpha}=Z_{\alpha\beta}\times_{S_{\alpha\beta}} S_{\alpha}$. It is enough
to show that $Z_\lambda=Z_\alpha\times_{S_\alpha} S_\lambda\to S_\lambda$ is
proper for sufficiently large $\lambda$. Thus, property~\ref{PI:limit:proper}
proper, follows from the strongly representable case.

\ref{PI:limit:aff}--\ref{PI:limit:finite}: If $f$ is affine (resp.\
quasi-affine), then $f$ factors as a closed immersion (resp.\ immersion) $X\inj
\A{n}_S \to S$. As $\A{n}_S=\A{n}_{S_0}\times_{S_0} S$, it follows, by
Proposition~\pref{P:std-limits}, that for sufficiently large $\lambda$ there is
a factorization $X_\lambda\to \A{n}_{S_\lambda} \to S_\lambda$ such that $X\inj
\A{n}_S$ is a pull-back of $X_\lambda\to \A{n}_{S_\lambda}$. For sufficiently
large $\lambda$, the latter morphism is a closed immersion (resp.\
immersion). If $f$ is finite, then $f_\lambda$ is affine for sufficiently
large $\lambda$ and we can apply the strongly representable case.

\ref{PI:limit:first-fiber}--\ref{PI:limit:last-fiber}: These properties
can be checked on fibers and are constructible by
Propositions~\pref{P:constructible-conds} and~\pref{P:constructible-conds-2}.
The result for these properties thus follows from~\cite[Prop.~9.3.3]{egaIV}.
%
\end{proof}

\begin{remark}[{(cf.\ \cite[Cor.~8.2.11]{egaIV})}]\label{R:std-limits:qc-open}
Let $X=\varprojlim_{\lambda} X_\lambda$ be a limit of quasi-compact and
quasi-separated algebraic stacks and let $U\subseteq X$ be
an open quasi-compact substack. Then, there exists an index $\lambda$ and an
open
quasi-compact substack $U_\lambda\subseteq X_\lambda$ such that
$U=U_\lambda\times_{X_\lambda} X$. This follows from the previous propositions
since an open quasi-compact immersion is of finite presentation.
\end{remark}

\end{section}


\bibliography{noetherian}

\providecommand{\bysame}{\leavevmode\hbox to3em{\hrulefill}\thinspace}
\providecommand{\MR}{\relax\ifhmode\unskip\space\fi MR }
\providecommand{\MRhref}[2]{%
  \href{http://www.ams.org/mathscinet-getitem?mr=#1}{#2}
}
\providecommand{\href}[2]{#2}
\begin{thebibliography}{EHKV01}

\bibitem[Alp13]{alper_good-mod-spaces}
Jarod Alper, \emph{Good moduli spaces for {A}rtin stacks}, Ann. Inst. Fourier
  (Grenoble) \textbf{63} (2013), no.~6, 2349--2402.

\bibitem[AOV08]{abramovich_olsson_vistoli_tame_stacks}
Dan Abramovich, Martin Olsson, and Angelo Vistoli, \emph{Tame stacks in
  positive characteristic}, Ann. Inst. Fourier (Grenoble) \textbf{58} (2008),
  no.~4, 1057--1091.

\bibitem[CLO12]{conrad-lieblich-olsson_Nagata}
Brian Conrad, Max Lieblich, and Martin Olsson, \emph{Nagata compactification
  for algebraic spaces}, J. Inst. Math. Jussieu \textbf{11} (2012), no.~4,
  747--814.

\bibitem[Con07]{conrad_nagata}
Brian Conrad, \emph{Deligne's notes on {N}agata compactifications}, J.
  Ramanujan Math. Soc. \textbf{22} (2007), no.~3, 205--257.

\bibitem[EGA\textsubscript{I}]{egaI_NE}
A.~Grothendieck, \emph{\'{E}l\'ements de g\'eom\'etrie alg\'ebrique. {I}. {L}e
  langage des sch\'emas}, second ed., Die Grundlehren der mathematischen
  Wissenschaften in Einzeldarstellungen, vol. 166, Springer-Verlag, Berlin,
  1971.

\bibitem[EGA\textsubscript{II}]{egaII}
\bysame, \emph{\'{E}l\'ements de g\'eom\'etrie alg\'ebrique. {II}. \'{E}tude
  globale \'el\'ementaire de quelques classes de morphismes}, Inst. Hautes
  \'Etudes Sci. Publ. Math. (1961), no.~8, 222.

\bibitem[EGA\textsubscript{IV}]{egaIV}
\bysame, \emph{\'{E}l\'ements de g\'eom\'etrie alg\'ebrique. {IV}. \'{E}tude
  locale des sch\'emas et des morphismes de sch\'emas}, Inst. Hautes \'Etudes
  Sci. Publ. Math. (1964-67), nos.~20, 24, 28, 32.

\bibitem[EHKV01]{edidin-etal_Brauer-groups-and-quotient-stacks}
Dan Edidin, Brendan Hassett, Andrew Kresch, and Angelo Vistoli, \emph{Brauer
  groups and quotient stacks}, Amer. J. Math. \textbf{123} (2001), no.~4,
  761--777.

\bibitem[Gro10]{gross_thesis}
Philipp Gross, \emph{Vector bundles as generators on schemes and stacks},
  Ph{D}. {T}hesis, D{\"u}sseldorf, May 2010.

\bibitem[Gro13]{gross_tensor-generators}
\bysame, \emph{Tensor generators on schemes and stacks}, Preprint, Jun 2013,
  \href{http://arXiv.org/abs/1306.5418}{\mbox{arXiv:1306.5418}}.

\bibitem[Knu71]{knutson_alg_spaces}
Donald Knutson, \emph{Algebraic spaces}, Springer-Verlag, Berlin, 1971, Lecture
  Notes in Mathematics, Vol. 203.

\bibitem[Kre13]{kresch_flat-strat}
Andrew Kresch, \emph{Flattening stratification and the stack of partial
  stabilizations of prestable curves}, Bull. Lond. Math. Soc. \textbf{45}
  (2013), no.~1, 93--102.

\bibitem[LMB00]{laumon}
G{\'e}rard Laumon and Laurent Moret-Bailly, \emph{Champs alg\'ebriques},
  Springer-Verlag, Berlin, 2000.

\bibitem[L{\"u}t93]{lutkebohmert_compactification}
W.~L{\"u}tkebohmert, \emph{On compactification of schemes}, Manuscripta Math.
  \textbf{80} (1993), no.~1, 95--111.

\bibitem[Ols05]{olsson_proper-coverings}
Martin Olsson, \emph{On proper coverings of {A}rtin stacks}, Adv. Math.
  \textbf{198} (2005), no.~1, 93--106.

\bibitem[Ols06]{olsson_Hom-stacks}
\bysame, \emph{{$\underline {\rm Hom}$}-stacks and restriction of scalars},
  Duke Math. J. \textbf{134} (2006), no.~1, 139--164.

\bibitem[Pay09]{payne_toric-vect-bundles-res-prop}
Sam Payne, \emph{Toric vector bundles, branched covers of fans, and the
  resolution property}, J. Algebraic Geom. \textbf{18} (2009), no.~1, 1--36.

\bibitem[Per76]{perrin_approx-groups}
Daniel Perrin, \emph{Approximation des sch\'emas en groupes, quasi compacts sur
  un corps}, Bull. Soc. Math. France \textbf{104} (1976), no.~3, 323--335.

\bibitem[Ray68]{raynaud_sem-samuel}
Michel Raynaud, \emph{Un crit\`ere d'effectivit\'e de descente}, S\'eminaire
  d'{A}lg\`ebre {C}ommutative dirig\'e par {P}ierre {S}amuel: 1967--1968. {L}es
  \'epimorphismes d'anneaux, Exp. No. 5, Secr\'etariat math\'ematique, Paris,
  1968, p.~22.

\bibitem[RG71]{raynaud-gruson}
Michel Raynaud and Laurent Gruson, \emph{Crit\`eres de platitude et de
  projectivit\'e. {T}echniques de ``platification'' d'un module}, Invent. Math.
  \textbf{13} (1971), 1--89.

\bibitem[Ryd10]{rydh_submersion_and_descent}
David Rydh, \emph{Submersions and effective descent of \'etale morphisms},
  Bull. Soc. Math. France \textbf{138} (2010), no.~2, 181--230.

\bibitem[Ryd11a]{rydh_embeddings-of-unramified}
\bysame, \emph{The canonical embedding of an unramified morphism in an \'etale
  morphism}, Math. Z. \textbf{268} (2011), no.~3-4, 707--723.

\bibitem[Ryd11b]{rydh_etale-devissage}
\bysame, \emph{\'{E}tale d\'evissage, descent and pushouts of stacks}, J.
  Algebra \textbf{331} (2011), 194--223.

\bibitem[Ryd11c]{rydh_hilbert}
\bysame, \emph{Representability of {H}ilbert schemes and {H}ilbert stacks of
  points}, Comm. Algebra \textbf{39} (2011), no.~7, 2632--2646.

\bibitem[SP]{stacks-project}
{The Stacks Project Authors}, \emph{Stacks project},
  \url{http://stacks.math.columbia.edu/}.

\bibitem[Sum75]{sumihiro_II}
Hideyasu Sumihiro, \emph{Equivariant completion. {II}}, J. Math. Kyoto Univ.
  \textbf{15} (1975), no.~3, 573--605.

\bibitem[Tem11]{temkin_relative-RZ-spaces}
Michael Temkin, \emph{Relative {R}iemann-{Z}ariski spaces}, Israel J. Math.
  \textbf{185} (2011), 1--42.

\bibitem[Tot04]{totaro_resolution-property}
Burt Totaro, \emph{The resolution property for schemes and stacks}, J. Reine
  Angew. Math. \textbf{577} (2004), 1--22.

\bibitem[TT90]{thomason-trobaugh}
R.~W. Thomason and Thomas Trobaugh, \emph{Higher algebraic {$K$}-theory of
  schemes and of derived categories}, The Grothendieck Festschrift, Vol.\ III,
  Progr. Math., vol.~88, Birkh{\"a}user Boston, 1990, pp.~247--435.

\end{thebibliography}
\bibliographystyle{dary}

\end{document}